\documentclass[a4paper,reqno]{amsart}
\usepackage[utf8]{inputenc}
\usepackage{amssymb}
\usepackage{enumitem}
\usepackage{mathrsfs}
\usepackage{mathtools}
\usepackage{exscale}
\usepackage{relsize}
\usepackage{tikz-cd}
\usepackage[all,pdf]{xy}
\setlist[enumerate]{label=\textnormal{(\roman*)}}
\usepackage{hyperref}
\usepackage{color}
\usepackage{verbatim}

\newtheorem{theorem}{Theorem}[section]
\newtheorem{corollary}[theorem]{Corollary}
\newtheorem{lemma}[theorem]{Lemma}
\newtheorem{proposition}[theorem]{Proposition}

\theoremstyle{definition}
\newtheorem{definition}[theorem]{Definition}
\newtheorem{remark}[theorem]{Remark}
\newtheorem{conjecture}[theorem]{Conjecture}

\numberwithin{equation}{section}
\newcommand\Pl{\mathbb{P}^{1}}
\newcommand{\OO}{\mathcal O}

\newcommand*{\dif}{\mathop{}\!\mathrm{d}}

\newcommand{\sC}{{\mathcal C}}
\newcommand{\sD}{{\mathcal D}}
\newcommand{\sE}{{\mathcal E}}
\newcommand{\sF}{{\mathcal F}}
\newcommand{\sG}{{\mathcal G}}

\newcommand{\sL}{{\mathcal L}}
\newcommand{\sM}{{\mathcal M}}

\newcommand{\sO}{{\mathcal O}}
\newcommand{\sP}{{\mathcal P}}
\newcommand{\sQ}{{\mathcal Q}}

\newcommand{\sS}{{\mathcal S}}
\newcommand{\sT}{{\mathcal T}}

\newcommand{\sV}{{\mathcal V}}

\newcommand{\A}{{\mathbb A}}

\newcommand{\F}{{\mathbb F}}

\newcommand{\N}{{\mathbb N}}
\renewcommand{\P}{{\mathbb P}}
\newcommand{\Q}{{\mathbb Q}}
\newcommand{\R}{{\mathbb R}}

\newcommand{\Z}{{\mathbb Z}}
\newcommand{\id}{{\rm id}}
\newcommand{\rank}{{\rm rank}}

\newcommand{\Gl}{{\rm Gl}}

\newcommand{\Pic}{\mathrm{Pic}}

\newcommand{\Gr}{\mathrm{Gr}}
\newcommand{\Gm}{\mathbb{G}_m}
\newcommand{\Spec}{\mathrm{Spec}}

\newcommand{\Char}{\textrm{char}}

\newcommand{\HIG}{\mathrm{HIG}}

\newcommand{\eg}{\textit{e}.\textit{g}.}
\newcommand{\etc}{\textit{etc}.}
\newcommand{\Par}{\textrm{par}}
\newcommand{\nil}{\textrm{nil}}
\newcommand{\Higgs}{\textnormal{Higgs}}
\newcommand{\pardeg}{\mathrm{pdeg}}
\begin{document}

\title[A torsion property of the zero of Kodaira-Spencer]
	{A torsion property of the zero of Kodaira-Spencer over $ \Pl $ removing four points}

	\author[Xiaojin Lin]{Xiaojin Lin}	
	\author[Mao Sheng]{Mao Sheng}
	\author[Jianping Wang]{Jianping Wang}
	
	\email{xjlin@mail.ustc.edu.cn}
	\email{msheng@ustc.edu.cn}
	\address{School of Mathematical Sciences, University of Science and Technology of China, Hefei, 230026, China.}
	\email{msheng@mail.tsinghua.edu.cn}
	\address{Yau Mathematical Science Center, Tsinghua University, Beijing, 100084, China.}
	\email{jianpw@mail.tsinghua.edu.cn}
	\address{Yau Mathematical Science Center, Tsinghua University, Beijing, 100084, China.}
	\begin{abstract}
			We establish a torsion theorem to the effect that the unique zero of the Kodaira-Spencer map attached to a certain quasi-semistable family of complex projective varieties over the complex projective line is the image of a torsion point of an elliptic curve under the natural projection. The proof is a mod $p$ argument and requires a density one set of primes. There are three essential ingredients in the proof: a solution to the conjecture of Sun-Yang-Zuo \cite{SYZ}, which constitutes the principal part of the paper, Pink's theorem \cite{Pink2004OnTO}, and Higgs periodicity theorem \cite{LS}\cite{ksdR}.
	\end{abstract}
\thanks{The work is supported by CAS Project for Young Scientists in Basic Research Grant No. YSBR-032, National Key Research and Development Project SQ2020YFA070080, National Natural Science Foundation of China (Grant No. 11721101), Fundamental Research Funds for the Central Universities.}

	\maketitle
	\tableofcontents
	\section{Introduction}
In his Lyon's talk \cite{Zuo} on a joint project with Jinbang Yang, Kang Zuo has envisioned an arithmetic theory of Higgs bundles over number fields, based on the notion of a periodic Higgs-de Rham flow \cite{sshiggs}. About the same time, the second named author and his collaborators \cite{LS}, \cite{ksdR}, \cite{kshiggs} have independently developed a basic theory of periodic Higgs/de Rham bundles over \emph{complex} algebraic curves, aiming at extending the periodicity in positive characteristic, first introduced in \cite{SZ12}, to the field of complex numbers. By solving a beautiful conjecture due to Sun-Yang-Zuo \cite{SYZ}, we shall present in this work a geometric application of the theory on the Kodaira-Spencer maps associated to certain quasi-semistable families over the complex projective line.   
	
	Let $C$ be a smooth projective complex curve, and $D\subset C$ a reduced effective divisor. 
	To a smooth projective morphism $f: X\to U$ with $U\subset C-D$, one attaches the $i$-th parabolic Gau{\ss}-Manin system for each $i\geq 0$, which is a parabolic de Rham bundle over $(C,D)$. 
	Its associated graded parabolic Higgs bundle is called a parabolic Kodaira-Spencer system over $(C,D)$. 
	A parabolic Higgs bundle over $(C,D)$ is said to be \emph{motivic} if it is a direct factor of some parabolic Kodaira-Spencer system (see \cite[Definition 3.4]{ksdR}). 
	\begin{definition}(\cite[Definition 4.7]{ksdR})\label{globally periodic}
		A parabolic Higgs bundle $(E,\theta)$ over $(C,D)$ is called \emph{periodic} if it is graded and there exists:
		\begin{itemize}
			\item an integer $f$;
			\item a spreading-out $$(\mathscr{C},\mathscr{D}, \mathscr{E},\Theta)\rightarrow S,$$ where $S$ is an integral scheme of finite type over $\Z$; and
			\item a proper closed subscheme $Z\subset S$,  
		\end{itemize} 
		such that for all geometric points $s\in S-Z$, the reduction $(\mathscr{E}_s,\Theta_s)$ at $s$ is $f$-periodic with respect to \emph{all} $W_2(k(s))$-lifting $\tilde s: \Spec(W_2(k(s)))\to S-Z$. 
	\end{definition}
	The Higgs periodicity theorem links motivicity to periodicity. 
	\begin{theorem}[Theorem 1.3 \cite{LS} , Theorem 5.10 \cite{ksdR}]\label{Higgs periodicity theorem}
		%\cite[Theorem 5.10]{ksdR}
		Motivic parabolic Higgs bundles over $(C,D)$ are periodic.  
	\end{theorem}
	Now let $D\subset \P^1_{\mathbb{C}}$ be an effective divisor consisting of four distinct points. 
	The following \emph{torsion theorem} is the main result of the paper.
	\begin{theorem}\label{torsion theorem}
		Let $(E,\theta)$ be a rank two motivic parabolic Higgs bundle over $(\P^1_{\mathbb{C}},D)$ which has a unique zero $x$ inside $\P^1_{\mathbb{C}}-D$ and has equal parabolic weights at each component of $ D $. 
		Then $x$ must be the image of a torsion point of the elliptic curve $C$ under $\pi: C\to \P^1_{\mathbb{C}}$, the double covering map to $\P^1_{\mathbb{C}}$ with branch locus $D$.
	\end{theorem}
	\begin{remark}
The torsioness theorem is essentially due to Jinbang Yang and Kang Zuo. Indeed, in Lyon's talk \cite{Zuo}, Kang Zuo has already predicted such a phenomenon, based on a calculation on 26 quasi-semistable families of elliptic curves over $\mathbb P^1$ with three semistable fibers and one quasi-semistable fiber, carried out by Xin Lv, Jun Lu and Jinbang Yang. It is clear to Yang and Zuo that, when $\lambda$ is defined over $\bar \Q$, the solution of Sun-Yang-Zuo conjecture (see below) plus Pink's theorem on characterizing torsion points of an abelian variety defined over $\bar \Q$, leads to the torsion theorem in the essential case of the theorem. Their calculation implies that the apparent singularities \footnote{The apparent singularity of the Gau{\ss}-Manin connection is nothing but the zero of its corresponding Kodaira-Spencer map. It is the zero of $ \delta $ \cite[Theorem 3.10]{Doran2001AlgebraicAG}, which is exactly the apparent singularity of Gau{\ss}-Manin connection by the computation in \cite[P.16]{Gontsov2009ApparentSO}} of Family 2-3 in Table 1 \cite{Doran2001AlgebraicAG} are torsion. 
\end{remark}

	We defer its proof until the end of the section. 
	Let $k$ be an algebraically closed field of char $p\geq 3$, let $\lambda\in k-\{0,1\}$.
	Form the reduced divisor $D=0+1+\lambda+\infty$ of projective line $\P^1_{k}$ over $ k $, and fix a $W_2(k)$-lifting $(\P^1_{W_2(k)},\tilde D)$ of the log pair $(\Pl_{k},D)$ over $k$.
	We abbreviate $ \Pl_{k} $ to $ \Pl $.
	Consider the following type of rank two graded logarithmic Higgs bundles $(E,\theta)$, where 
	\begin{equation*}
	E=E^{1,0}\oplus E^{0,1}=\sO_{\Pl} \oplus \sO_{\Pl}(-1),
	\end{equation*}
	and the Higgs field
	\begin{equation*}
	\theta: \sO_{\Pl} \to \sO_{\Pl}(-1) \otimes \Omega_{\P^{1}}(D)
	\end{equation*}
	is nonzero. 
	It is stable and the irreducible component of the moduli space of Higgs bundles containing the isomorphism class $[(E,\theta)]$ is simply a $\P^1$. 
	Indeed, the moduli space is isomorphic to  the complete linear system $|\sO_{\Pl}(-1)\otimes \Omega^1_{\Pl}(D)|=|\sO_{\Pl}(1)|\cong \P^1$. 
	The isomorphism class is labelled by the zero of $\theta$. 
	Thinking of $\P^1=\A^1\cup \{\infty\}$, we may write $\theta=\theta_x$ for $x$ is the zero of $\theta$ in its affine coordinate.

	In \cite{SYZ}, Sun-Yang-Zuo introduced the notion of \emph{twisted periodic Higgs-de Rham flow}, and as an illustration of their theory, they showed that the flow induces a self map 
	\begin{equation*}
	\varphi_{\lambda,p}: \Pl(k)\to \Pl(k)
	\end{equation*}
	as follows: 
	Let $C^{-1}$ be the inverse Cartier transform from the category of rank two nilpotent logarithmic Higgs bundles over $(\Pl,D)$ to the category of logarithmic connections over $(\Pl,D)$ (with respect to the chosen $W_2$-lifting). 
	Let $\Gr$ be the grading functor with respect to the \emph{Harder-Narasimhan} filtration of the underlying vector bundle. They deduced the following isomorphism:
	\begin{equation}\label{twist by O(m)}
	\Gr\circ C^{-1}(\sO_{\Pl} \oplus \sO_{\Pl}(-1),\theta_x)\cong (\sO_{\Pl} \oplus \sO_{\Pl}(-1),\theta_{x'})\otimes (\sO_{\Pl}(m),0),
	\end{equation}
	where $m=\frac{1-p}{2}$. 
	Ignoring the uninteresting rank one twist factor, one obtains another rank two graded logarithmic Higgs bundle of the same type by the flow. Sun-Yang-Zuo in loc. cit. showed there is an inseparable endomorphism $\varphi_{\lambda,p}$ of $\Pl$ of degree $p^2$ mapping $x$ to $x'$ and make the following 
	
	\begin{conjecture}[Sun-Yang-Zuo, Conjecture 4.8 \cite{SYZ}]\label{SYZ conj}
		Let $C_{\lambda}$ be the elliptic curve defined by the affine equation $y^2=x(x-1)(x-\lambda)$ with the natural projection $(x,y)\mapsto x$, realizing $C_{\lambda}$ as a double cover of $\P^1$ branched along $D$. 
		Then the following diagram commutes:
		\begin{equation}\label{conjecture}
		\begin{gathered}
		\xymatrix{
			C_{\lambda} \ar[r]^{[p]} \ar[d]_{\pi} & C_{\lambda} \ar[d]^{\pi} \\
			\Pl \ar[r]^{\varphi_{\lambda,p}} & \Pl, 
		}
		\end{gathered}  
		\end{equation}
		where $[p]$ is the multiplication by $p$ map of the elliptic curve. 
	\end{conjecture}
	They verified the conjecture for $ p<50 $. 
	Now we may confirm the conjecture up to a sign.
	\begin{theorem}\label{main result}
		Notation as in Conjecture \ref{SYZ conj}. Then the above diagram commutes up to sign, viz. 
		\begin{equation*}
		\pi \circ \pm [p]=\varphi_{\lambda,p}\circ \pi.
		\end{equation*}
	\end{theorem}
	Note that, as far as the properties of $\varphi_{\lambda,p}$ are concerned, it is irrelevant to determine the sign. The idea of the proof is to give a moduli interpretation of the diagram \eqref{conjecture}. See the diagram \eqref{commutitivity of flow operator}. To each $P\in C_{\lambda}(k)$, we associate the stable logarithmic Higgs bundle $F(P)=(\sO_{C_{\lambda}}(P)\oplus \sO_{C_{\lambda}}(-P),\eta_{2\sigma(P)})$, where the Higgs field
	\begin{equation*}
	\eta_{2\sigma(P)}: \sO_{C_{\lambda}}(P) \to \sO_{C_{\lambda}}(-P)\otimes \Omega_{C_{\lambda}}(D)
	\end{equation*}
	has a unique zero at the involution $\sigma(P)$ of $P$ with multiplicity two. 
	It is easy to see that the submoduli of $F(P)$ keeping the structure is naturally isomorphic to $C_{\lambda}$. Next, we construct a natural transformation between moduli functors which gives rise to a global morphism $C_{\lambda}\to \P^1$ of moduli spaces, which sends the class $[F(P)]$ of $F(P)$ to the class $[E(\pi(P))]$ of $E(\pi(P)):=(\sO_{\Pl} \oplus \sO_{\Pl}(-1),\theta_{\pi(P)})$. The hard part is to show this construction is compatible with the flow operator. But before that one needs first to show the flow operator over $C_{\lambda}$ exists! Interestingly enough, once this has been done, the remaining ingredient for the commutativity is just the functoriality of the flow operator (Lemma \ref{parabolic pullback commutes with flow operator}) (which one would expect to use anyway). To this end, we construct a \emph{flow morphism} (Definition \ref{flow morphism on the whole moduli}) for the whole moduli (of semistable graded logarithmic Higgs bundles) containing $F(P)$, which contains several connected components. Because it is a morphism, instead of being merely a set-theoretical self map, the geometry of the moduli space applies, which indeed allows us to deduce that the action of the flow operator on the whole moduli preserves this submoduli $C_{\lambda}$. Hence we also obtain a self map (Corollary \ref{flow morphsim preserving Sigma})
	$$
	\phi_{\lambda,p}: C_{\lambda}\to C_{\lambda},
	$$
	and simultaneously we establish the key formula (Proposition \ref{decent})
	\begin{equation}\label{key formula}
	\pi(\phi_{\lambda,p}([F(P)]))=\varphi_{\lambda,p}([E(\pi(P))]). 
	\end{equation}
	It is relatively easy to conceive that $\phi_{\lambda,p}=\pm [p]$ (Lemma \ref{arithmetic lemma}). After all, by rigidity lemma (\cite[Proposition 6.1]{MFK}), any nonconstant endomorphism of $C_{\lambda}$ is an isogeny up to translation.

 Now we may proceed to the proof of Theorem \ref{torsion theorem}.
	\begin{proof}
		We may assume $D=0+1+\lambda+\infty$ for some $\lambda\in \mathbb{C}$. Let $(E,\theta)$ be a motivic parabolic Higgs bundle over $(\P^1_{\mathbb{C}},D)$ satisfying the condition in Theorem \ref{main result}. By Theorem \ref{Higgs periodicity theorem}, it is periodic. Because it has equal parabolic weights at any component of $D$, it follows from Proposition \ref{identification of self-maps} that we may simply assume that $(E,\theta_x)=(\sO_{\Pl_{\mathbb{C}}} \oplus \sO_{\Pl_{\mathbb{C}}}(-1),\theta_{x})$, where $x\in \mathbb{C}$ is the unique zero of $\theta$. Let $\pi: C_{\lambda}\to \Pl_{\mathbb{C}}$ be the double cover branching along $D$. Pick $P\in C_{\lambda}$ such that $\pi(P)=x$ and form 
		\begin{equation*}
		F(P)=(\sO_{C_{\lambda}}(P) \oplus \sO_{C_{\lambda }}(-P),\eta_{2\sigma(P)}),
		\end{equation*}
		as we did above in the char $p$ situation. 
		Write $E(\pi(P))$ for $(E,\theta_x)$. Then the formula \eqref{key formula} implies that $F(P)$ is periodic if and only if $E(\pi(P))$ is periodic. 
		But since the flow operator for $F(P)$ in char $p$ is simply $\pm [p]$ after Lemma \ref{arithmetic lemma}, it follows from Pink's theorem \cite[Theorem 5.3]{Pink2004OnTO} and the main theorem of Masser \cite{Masser1989SpecializationsOF} that $P$ has to be torsion (see the proof of \cite[Proposition 2.5]{kshiggs} for more details). 
		Hence our torsion theorem follows. 
	\end{proof}
	To conclude the introduction, we would like to make a generalization of the periodic Higgs conjecture (\cite[Conjecture 4.1]{kshiggs}) as follows: 
  	\begin{conjecture}\label{parabolic PHC}
	Let $C$ be a smooth projective curve over $\mathbb C$, and $D\subset C$ a reduced effective divisor. 
	Then any stable periodic parabolic Higgs bundle over $(C,D)$ is motivic.     
  	\end{conjecture}
   For the rank one case with empty $D$, the conjecture holds, as proved in \cite[Proposition 2.5]{kshiggs}. Here is the extension of that result to the general case.
   \begin{proposition}\label{PHC for rank one}
   Let $(C,D)$ be as in Conjecture \ref{parabolic PHC}. Then for a parabolic Higgs line bundle $(L,\theta)$ over $(C,D)$ is periodic if and only if $\theta=0$ and $L$ is torsion, that is, there exists a positive integer $m$ such that $L^{\otimes m}\cong \sO_C$ as parabolic line bundles. Moreover, all rank one parabolic Higgs bundles are motivic.  
   \end{proposition}
   \begin{proof}
   Let $L$ be a torsion parabolic line bundle viz. $L^{\otimes m}\cong \sO_C$. Take a spread-out $(\sC,\sD,\sL)/S$ of $(C,D,L)/\mathbb C$. For any geometric closed point $s\in S$ with $\mathrm{char}(k(s))$ large enough, it holds that
   $$
   \mathrm{Gr}_{Fil_{tr}}\circ C^{-1}_{\Par}(\sL_s,0)\cong (\sL_s^{\otimes p},0),
   $$
and hence $(\mathrm{Gr}_{Fil_{tr}}\circ C^{-1}_{\Par})^{\phi(m)}(\sL_s,0)\cong (\sL_s,0)$. So $(L,0)$ is periodic. Conversely, one takes a suitable cyclic cover $\pi: C'\to C$ whose Galois group is $G$ and whose branch locus contains $D$, such that $(M,\eta):=\pi_{\Par}^{*}(L,\theta)$ has the trivial parabolic structure. It follows from Lemma \ref{parabolic pullback commutes with flow operator} (see \S2) that $(M,\eta)$ is a periodic logarithmic Higgs bundle over $C'$. By \cite[Proposition 2.5]{kshiggs}, $M$ is torsion and $\eta=0$. It follows that $\theta=0$, and there is some $m$ such that one has an isomorphism of $G$-equivariant line bundles:
$$
\pi_{\Par}^{*}(L^{\otimes m})\cong (\pi_{\Par}^{*}L)^{\otimes m}\cong \sO_{C'}.
$$
Taking more tensor powers, we may even assume that in the above isomorphism, $\sO_{C'}$ is equipped with the trivial $G$-structure. It follows that
$$
L^{\otimes m}\cong \pi_{\Par*}(\sO_{C'})\cong \sO_{C}
$$
as parabolic line bundles. 

Finally, we show the motivicity for a torsion parabolic line bundle. As $M$ is torsion, there is a finite \'etale cyclic cover $\pi': C''\to C'$ so that $\pi'^*M\cong \sO_{C''}$. Taking the invariant part of $\pi'_{*}(\sO_{C''},d)$ under the cyclic group action, we obtain a connection $(M,\nabla)$ over $C'$. The direct sum $(M,\nabla)[G]:=\oplus_{g\in G}g^*(M,\nabla)$ admits a natural $G$-action. Notice that, as $M$ is $G$-equivariant, there is an isomorphism of $G$-equivariant bundles $\oplus_{g\in G}g^*M\cong M^{\oplus |G|}$. Therefore, $\pi_{\Par*}\oplus_{g\in G}g^*M\cong L^{\oplus |G|}$. Consequently, there is a parabolic connection $\nabla$ on $L^{\oplus |G|}$. Clearly, $(L^{\oplus |G|},\nabla,Fil_{tr})$ is a motivic parabolic de Rham bundle. Therefore $(L^{\oplus |G|},0)=\Gr_{Fil_{tr}}(L^{\oplus |G|},\nabla)$ is motivic. So $(L,0)$ is motivic too. 
\end{proof}
For a rank two Higgs bundle $(E,\theta)$ over $(\P_{\mathbb C}^1,0+1+\lambda+\infty)$ considered above, we equip it with a parabolic structure along $0+1+\lambda+\infty$ with equal parabolic weights at each point and with parabolic degree zero. By abuse of notation, the resulting parabolic Higgs bundle is denoted by $(E,\theta)$. Then the periodic Higgs conjecture predicts that once the zero of $\theta$ is torsion in the sense of Theorem \ref{torsion theorem}, it must be motivic. In a very recent paper \cite{YZ}, Jinbang Yang and Kang Zuo proved the following deep result (\cite[Theorem 1.6]{YZ}).
 \begin{theorem}[Yang-Zuo]\label{motivicity theorem}
Notation as above. Suppose the parabolic structure is non-trivial only at one point. Then $(E,\theta)$ is motivic.
 \end{theorem}
In fact, $(E,\theta)$ arises from families of Abelian varieties by Yang-Zuo's construction. Combining Theorem \ref{motivicity theorem} with Proposition \ref{PHC for rank one}, Yang-Zuo's motivicity theorem extends to  arbitrary equal parabolic weights. In other words, Conjecture \ref{parabolic PHC} indeed holds in this special case.
 
The paper is organized as follows. In \S2, we lay down basics of the theory of parabolic Higgs-de Rham flow, to the extent so that the flow operator $\varphi_{\lambda,p}$ of Sun-Yang-Zuo can be reinterpreted via this theory. \S3-\S5 form the main body of the paper, which gives the diagram \eqref{conjecture} a moduli interpretation. As a result, we obtain the commutative diagram \eqref{commutitivity of flow operator}. In \S6, we prove Conjecture \ref{SYZ conj} by furnishing the proof that the flow operator $\phi_{\lambda,p}$ over the elliptic curve $C_{\lambda,p}$ is indeed the multiplication by $p$ map, up to sign. 
	The efforts towards understanding the flow operator $\varphi_{\lambda,p}$ of Sun-Yang-Zuo allow us to obtain a clean classification result in \S7. 
	Their isomorphism classes form the set of \emph{periodic points} in the corresponding moduli spaces of graded parabolic Higgs bundles. 
	We carry out some rudimentary study of its distribution in the moduli.

	\textbf{Acknowledgement:} This paper owes its existence to the lectures on the work \cite{SYZ} and its subsequent development, delivered by Kang Zuo in November 2021 at USTC. 
	We would like to thank him heartily for his constant encouragement and support. Our thanks also go to members of the USTC algebraic geometry team, especially Jinbang Yang, whose comments and suggestions are greatly appreciated. In particular, he pointed out a serious error in our first calculation of the parabolic pushforward of $(\sF,\eta)[G]$ (see \S4), which is highly valuable to us. We thank Jiangwei Xue for helpful discussions on isogenies of elliptic curves in positive characteristic, especially providing us a proof of Lemma \ref{arithmetic lemma} in the supersingular case. 
	Last but not least, we thank Shing-Tung Yau, who drew our attention to the work \cite{Doran2001AlgebraicAG}.
	
	\section{Parabolic Higgs-de Rham flow}
	Let $k$ be an algebraically closed field. Let $C$ be a smooth projective curve over $k$, equipped with a reduced effective divisor $D=\sum_iD_i$. 
	A parabolic connection (resp. Higgs bundle) $V$ over $(C,D)$ is a pair $(V,\nabla)$ (resp. $(E,\theta)$), where $V$ (resp. E) is a parabolic vector bundle over $(C,D)$, and $\nabla$ (resp. $\theta$) is logarithmic connection (resp. logarithmic Higgs field) over $V_0$ (resp. $E_0$) preserving the parabolic filtration on $V$ (resp. $E$). 
	See \cite[Definition 2.7]{ksdR} and the discussion following it. 
	They form tensor categories. 
	It is useful to observe that there is an obvious isomorphism in the category of parabolic vector bundles over $(C,D)$
	\begin{equation}\label{tensor product}
	V(-\sum_i\alpha_iD_i) \cong  V\otimes \sO_{C}(-\sum_i\alpha_iD_i),
	\end{equation}
	where $V(-\sum_i\alpha_iD_i)$ is defined in \cite[Example 2.5]{ksdR}, $V$ is a parabolic vector bundle equipped with the trivial parabolic structure \cite[Example 2.4]{ksdR}. 

	Let us focus on the case $\textrm{char}(k)=p>0$ for a while, so that we may speak about a parabolic Higgs-de Rham flow over $(C,D)/k$. 
	For that, we also need to fix a $W_2(k)$-lifting of the pair $(C,D)$. In \cite[Theorem 2.10]{ksdR}, the parabolic inverse Cartier transform $C_{\textrm{par}}^{-1}$ is constructed:
	\begin{equation*}
	C_{\textrm{par}}^{-1}: \textrm{HIG}^{\textrm{par}}_{p-1,N}((C,D)/k)\to \textrm{MIC}^{\textrm{par}}_{p-1,N}((C,D)/k), 
	\end{equation*}
	where $N$ is a positive integer coprime to $p$, $\textrm{HIG}^{\textrm{par}}_{p-1,N}((C,D)/k)$ is the category of parabolic Higgs bundles over $(C,D)/k$ which are nilpotent of exponent $\leq p-1$ and whose parabolic structures are supported in $D$ and whose weights are contained in $\frac{1}{N}\Z$, and $\textrm{MIC}^{\textrm{par}}_{p-1,N}((C,D)/k)$ is the category of adjusted \footnote{The adjusted condition is defined in \cite[Definitiion 2.9]{ksdR}, which generalizes \cite[Definition 3.2]{IS} in characteristic zero. 
	We remark the following: 
	Let $(C,D)$ be defined over $\mathbb C$, and $(V^0,\nabla^0)$ be a flat connection defined over the complement $C-D$ with quasi-unipotent local monodromies. 
	Then the Deligne's canonical extension $(V,\nabla)$ of $(V^0,\nabla^0)$ is an adjusted parabolic connection by \cite[Lemma 3.3]{IS}. 
	Let $(\sC,\sD,\sV,\nabla)/S$ be a spread-out of $(C,D,V,\nabla)$. 
	Then away from a proper closed subset $Z\subset S$, the reduction of $(\mathcal C,\mathcal D,\mathcal V,\nabla)$ over $s\in S-Z$ is adjusted.} parabolic flat bundles over $(C,D)/k$ whose $p$-curvatures as well as the nilpotent part of residues are nilpotent of exponent $\leq p-1$ and whose parabolic structures are supported in $D$ and whose weights are contained in $\frac{1}{N}\Z$. 
	It is an equivalence of categories. Here is a reminder: Consider the following Cartesian diagram:
	$$
 \xymatrix{C'\ar[r]^-{F_k}\ar[d]_-{}&C\ar[d]^-{}\\
		k\ar[r]^-{F_k}&k,}
  $$
where $F_k: k\to k$ is the absolute Frobenius. Then the functor $C^{-1}$ as above differs by an obvious base change functor from the original one of Ogus-Vologodsky \cite{OV} (for $D=\emptyset$) , which sends a Higgs module over $C'$ to a module with connection over $C$.     
 
	Next, we shall say a few words about the grading functor in the parabolic setting. In \cite[Lemma 3.2]{ksdR}, (a version of) the grading functor is defined, which is a functor from the category of parabolic de Rham bundles over $(C,D)/k$ (\cite[Definition 3.1]{ksdR}) to the category of parabolic graded Higgs bundles over $(C,D)/k$. The point is that the filtration $Fil$ is defined \emph{only} on $V_0$, which induces a filtration on each term in the parabolic filtration of $V$ (which is called the \emph{induced filtration} in \cite[Lemma 3.2]{SW}). 
	Recall the following 
	\begin{definition}[Definition 3.3 \cite{SW}]\label{defn grsemistable}
	Let $(V,\nabla)$ be a parabolic connection over $(C,D)/k$. 
	A filtration $Fil$ over $(V,\nabla)$ is gr-(semi)stable, if it is a saturated and Griffiths transverse filtration of finite level over $(V_0,\nabla)$ such that the associated graded parabolic Higgs bundle $\Gr_{Fil}(V,\nabla)$ is (semi)stable.  
	\end{definition}
	The following is an existence and uniqueness result about gr-semistable filtrations. The existence part is a special case of \cite[Proposition 3.4]{SW}. 
	The uniqueness part generalizes \cite[Corollary 4.2]{Sim08},\cite[Corollary 5.6]{Lan14}, \cite[Lemma 7.1]{sshiggs}.
	\begin{proposition}\label{grsemistable}
	Let $k$ be an algebraically closed filed (with arbitrary characteristic). Let $(V,\nabla)$ be a semistable parabolic connection over $(C,D)/k$. Then the following statements hold:
	\begin{itemize}
    \item [(i)] There exists a gr-semistable filtration $Fil$ over $(V,\nabla)$ in the sense of Definition \ref{defn grsemistable}. 
    \item [(ii)] Let $Fil_i,i=1,2$ are gr-semistable filtrations on $(V,\nabla)$. Then $\Gr_{Fil_1}(V,\nabla)$ is $S$-equivalent to $\Gr_{Fil_2}(V,\nabla)$ (as parabolic Higgs bundle). Moreover, a gr-stable filtration on $(V,\nabla)$ is unique up to a shift of indices. 
	\end{itemize}
	\end{proposition}
	\begin{proof}
	For (i), we refer our reader to \cite[Proposition 3.4]{SW}, which adapts the proof of \cite[Theorem A.4]{sshiggs} to the parabolic setting. For (ii), one considers the moduli space $M_{\Lambda}$, corepresenting the moduli functor $\sM_{\Lambda}$ of semistable parabolic $\lambda$-connections over $(C,D)/k$. It is separated over $k$. There is one canonical $\Gm$-action on $M_{\Lambda}$: For $t\in \Gm(k)$ and $(E,\nabla)$, a parabolic $\lambda$-connection, the action is given by 
 $$
 [(E,\nabla)]\mapsto [(E,t\nabla)].
 $$
 For the connection $(V,\nabla)$ in the statement, which defines a point $x$ in the moduli, the $\Gm$-action defines an object $(\sV:=\textrm{pr}_C^*V,t\nabla)$ in $\sM_{\Lambda}(\Gm)$, whose associated moduli map $\phi: \Gm\to M_{\Lambda}$ has the initial value $\phi(1)=x$. Now we pick any gr-semistable filtration $Fil$ by (i). Claim: $[\lim_{t\to 0}(V,t\nabla)]=[\Gr_{Fil}(V,\nabla)]$. Then by the separatedness of the moduli, it follows that
 $$
 [\Gr_{Fil_1}(V,\nabla)]=[\Gr_{Fil_2}(V,\nabla)]
 $$
 for any two gr-semistable filtrations $Fil_i, i=1,2$, that is, they are $S$-equivalent. Note that the grading is nothing but the fiber over $0\in \A^1(k)$ of the Rees bundle 
 $$
 \xi(V,Fil)=\sum_i t^{-i}Fil^i\otimes \sO_{C_{\A^1}}\subset \textrm{pr}_C^*V,
 $$ 
equipped with the $t$-connection $t\nabla$. However, over the open part $\Gm\subset \A^1$, as $t$ is a unit, $(\xi(V,Fil)|_{\Gm},t\nabla)$ is naturally isomorphic to $(\sV,t\nabla)$. The claim is proved using the valuative criterion of separatedness. The proof for the uniqueness of a gr-stable filtration in Lemma \cite[Lemma 7.1]{sshiggs} works verbatim in the parabolic setting. So we omit the detail.
 \end{proof}		
	Following the procedure carried in \cite[Theorem A.4, Remark A.9]{sshiggs}, we obtain a \emph{unique} gr-semistable filtration, which we call the \emph{Simpson filtration} (denoted as $Fil_S$). 
	Remark that the Simpson filtration in rank two case is nothing but the Harder-Narasimhan (HN) filtration. 

	After constructing the functors $C^{-1}_{\textrm{par}}$ and $\textrm{Gr}_{Fil}$, one may define the notion of a periodic parabolic Higgs-de Rham flow over $(C,D)/k$. 
	We refer our reader to \cite[Definition 4.1]{ksdR} for a precise formulation. 
	A periodic parabolic Higgs bundle over $(C,D)/k$ is a graded parabolic Higgs bundle over $(C,D)/k$ initializing a perodic parabolic Higgs-de Rham flow. 
	The next is a rigidity result.
	\begin{proposition}\label{rigidity}
	Let $(E,\theta)$ be a periodic parabolic Higgs bundle over $(C,D)/\mathbb C$ in the sense of Definition \ref{globally periodic}. 
	Suppose it is stable. 
	Let $(\sC,\sD,\sE,\Theta)/S$ be a spread-out. 
	Then there exists a proper closed subset $Z\subset S$ and a positive integer $f$ such that for any geometric point $s\in S-Z$ and any $W_2(k(s))$-lifting $\tilde s: \Spec(W_2(k(s)))\to S-Z$ of $s$,
	\begin{equation*}
	(\Gr_{Fil_S}\circ C^{-1}_{\textnormal{par},s\subset \tilde s})^{f}(\sE_s,\Theta_s) \cong (\sE_s,\Theta_s).
	\end{equation*}
	\end{proposition}
	\begin{proof}
	All except that the grading is with respect to the Simpson filtration follows from the definition. 
	We note that the proof of \cite[Lemma 2.4]{LS} works verbatim in the parabolic setting and it shows that each intermediate Higgs term in a periodic parabolic Higgs-de Rham flow initializing a stable parabolic Higgs bundle is stable. 
	Since the stability is an open condition, we may conclude the proof by using Proposition \ref{grsemistable} (ii).   
	\end{proof}
	It has the following geometric consequence. 
	\begin{corollary}
	Let $(E,\theta)$ be a motivic parabolic Higgs bundle over $(C,D)/\mathbb C$. 
	Suppose it is stable \footnote{In general, it is polystable (\cite[Lemma 5.4]{ksdR}).}. 
	Let $(\sC,\sD,\sE,\Theta)/S$ be a spread-out. 
	Then there must exist a proper closed subset $Z\subset S$ and a positive integer $f$, such that for any geometric point $s\in S-Z$ and any $W_2(k(s))$-lifting $\tilde s: 
	\Spec(W_2(k(s)))\to S-Z$ of $s$, the isomorphism class $[(\sE_s,\Theta_s)]$ in its moduli space of semistable parabolic Higgs bundles over $(\sC_s,\sD_s)/k(s)$ of degree zero is fixed by the operator $(\textrm{Gr}_{Fil_S}\circ C^{-1}_{\textnormal{par},s\subset \tilde s})^{f}$.
	\end{corollary}
 	\begin{proof}
	Combine Theorem \ref{Higgs periodicity theorem} and Proposition \ref{rigidity}.
 	\end{proof}
	Therefore, in the study of motivic parabolic Higgs bundles, we are led to study the action of the flow operator $\Gr\circ C_{\textrm{par}}^{-1}$ on a moduli space of semistable parabolic Higgs bundles of degree zero in positive characteristic, where $\Gr$ refers to the grading functor $\Gr_{Fil_S}$. 
	Because $C_{\textrm{par}}^{-1}$ restricts to the inverse Cartier transform of Ogus-Vologodsky over objects with trivial parabolic structure, we may simply write $C_{\textrm{par}}^{-1}$ by $C^{-1}$ without causing much confusion. 
 
 A basic functorial property of the flow operator $\Gr\circ C^{-1}$ used in this paper is given as following. Let $k$ be an algebraically closed field of characteristic $p>0$. Let $N$ be a positive integer coprime to $p$. Let $C$ be a smooth projective curve over $k$, and let $\pi: C'\to C$ be a cyclic cover of order $N$, with a reduced branch divisor $D$. Let $G=<\sigma>\cong \Z/N\Z$ be the constant group scheme over $k$ acting on $C'$. Let $D'\subset C'$ be the reduced effective divisor formed by the pre-image of $D$. Then $\pi: (C',D')\to (C,D)$ is a log \'etale morphism over $k$. Recall the Biswas-Iyer-Simpson (BIS) correspondence for $\lambda$-connections ($\lambda=0,1$) established in \cite[Propsoition 2.14]{ksdR}: The parabolic pushforward $\pi_{\Par*}$ and the parabolic pullback $\pi_{\Par}^*$ induce an equivalence of categories between the category of $G$-equivariant logarithmic $\lambda$-connections over $(C',D')/k$ and the category of parabolic $\lambda$-connections over $(C,D)/k$ with parabolic weight in $\frac{1}{N}\Z$. Moreover, the functors are exact, and commute with direct sum and tensor product (see \cite[Remark 2.15]{ksdR}). For a $k$-scheme $T$, set $C_T=C\times_kT$ etc.. We also remark that the proof of the BIS correspondence in loc. cit. works verbatim in the relative setting, viz. for $G$-equivariant logarithmic $\lambda$-connectsions over $(C'_{T},D'_{T})/T$ and parabolic $\lambda$-connections over $(C_T,D_T)/T$. Here the $G$-action on $C'_T$ is given by
 $$
 \sigma_T:=\sigma\times \id_T: C'\times T\to C'\times T.
 $$
Choosing a $W_2(k)$-lifting  $(\tilde C,\tilde D)$ of $(C,D)$, there is a unique log \'etale morphism  $(\tilde C',\tilde D')\to (\tilde C,\tilde D)$ over $W_2(k)$ which lifts $\pi$.    
 \begin{lemma}\label{parabolic pullback commutes with flow operator}
 Notation as above. Let $(E,\theta)$ be an object in $\HIG^{\Par}_{p-1,N}((C,D)/k)$. Then there is a natural isomorphism:
 $$
 \Gr\circ C_{(C',D')\subset (\tilde C',\tilde D')}^{-1}(\pi_{\Par}^*(E,\theta))\cong  \pi_{\Par}^*\Gr\circ C_{(C,D)\subset (\tilde C,\tilde D)}^{-1}(E,\theta).
 $$
 \end{lemma}
\begin{proof}
By \cite[Lemma 2.19]{ksdR}, we have
$$
C^{-1}\pi_{\Par}^*(E,\theta)\cong \pi_{\Par}^*C^{-1}(E,\theta):=(V,\nabla).
$$
Since the HN filtration on the $G$-equivariant bundle $V$ is preserved under the $G$-action (where $G$ is the cyclic group of order $N$), it follows from the BIS correspondence that the parabolic pullback of the HN filtration on the parabolic bundle part of $C^{-1}(E,\theta)$ is nothing but the HN filtration of $V$. By Proposition \cite[Proposition 3.5]{ksdR}, 
$$
\Gr\circ C^{-1}\pi_{\Par}^*(E,\theta)\cong \pi_{\Par}^*\Gr\circ C^{-1}(E,\theta).
$$
The lemma is proved.
\end{proof}	
To conclude this section, we would like to point out one subtlety of the flow operator in the parabolic setting: The quasi-parabolic structure changes in general under the flow operator because of \cite[Proposition 2.18]{ksdR}. Therefore, the flow operator does \emph{not} necessarily induce a self map on the moduli space (with a fixed quasi-parabolic structure). However, in the case of equal parabolic weights, we may circumvent this problem by the following procedure: 
	Let $(E,\theta)$ be a (graded) parabolic Higgs bundle over $(C,D)/k$, where the parabolic weights of $E$ at each $D_i$ are equal. 
	Let $0\leq \alpha_i<1$ be the parabolic weight of $E$ at $D_i$. 
	Let $(\sO_C(\sum_i\alpha_iD_i),0)$ be the rank one parabolic Higgs bundle. 
	Then by \eqref{tensor product}, we get
	\begin{equation*}
	(E,\theta) \cong 
	(E_0,\theta ) \otimes ( \sO_C(\sum_{i}\alpha_iD_i),0 ) ,
	\end{equation*}
	where we equip $E_0$ with the trivial parabolic structure. 
	Clearly, $(E,\theta)$ is (semi)stable iff $(E_0,\theta)$ is (semi)stable. 
	Therefore, there is an obvious identification between the moduli space of semistable parabolic Higgs bundles containing  $(E,\theta)$ and the moduli space of semistable Higgs bundles containing $(E_0,\theta)$. 
	\begin{lemma}\label{tensor product lemma}
	Notation as above. 
	Suppose $\{\alpha_i\}_i\subset \frac{1}{N}\Z$, $\Char(k)=p$ with $(p,N)=1$. Then it holds that
	\begin{equation*}
	\Gr\circ C^{-1}(E,\theta)\cong \Gr\circ C^{-1}(E_0,\theta)\otimes (\sO_C(\sum_i p\alpha_iD_i),0).
	\end{equation*}
	Consequently, the parabolic weights of $\Gr\circ C^{-1}(E,\theta)$ at $D_i$ are equal of weight $\langle p\alpha_i\rangle$, where $\langle x\rangle$ is the fractional part of $x\in \R$.
	\end{lemma}
	\begin{proof}
	By \cite[Lemma 2.11]{ksdR}, it follows that
	\begin{equation*}
	C^{-1}(E,\theta)\cong C^{-1}(E_0,\theta)\otimes C^{-1}(\sO_C(\sum_i\alpha_iD_i),0).
	\end{equation*}
	Note that 
	\begin{equation*}
	C^{-1}(\sO_C(\sum_i\alpha_iD_i),0)\cong (\sO_C(\sum_i p\alpha_iD_i),\nabla_{can}).
	\end{equation*}
	It has parabolic weight $\langle p\alpha_i\rangle$ at $D_i$. Since it is stable of rank one, it follows that the Simpson filtration of $C^{-1}(E,\theta)$ is isomorphic to the tensor product of the Simpson filtration of  
	$C^{-1}(E_0,\theta)$ and $\sO_C(\sum_i\langle p
\alpha_i\rangle D_i)$. The rest is clear. 
	 \end{proof}
	The previous formula provides a useful way to understand the theory of twisted Higgs-de Rham flow \cite{SYZ}. 
	Let us see this through the example in the introduction. So $(C,D)=(\P^1,0+1+\lambda+\infty)$, and the graded logarithmic Higgs bundle $(E,\theta)=(\sO\oplus \sO(-1),\theta)$ is equipped with the trivial parabolic structure. 
	Because $\pardeg(E)=\deg(E)=-1$, it cannot be motivic after Theorem \ref{Higgs periodicity theorem} and \cite[Corollary 4.9]{ksdR}. 
	It means that we must put a nontrivial parabolic structure on $E$ to match the parabolic degree that is forced to be zero by periodicity. 
	There are infinitely many choices to do so. 
	In conform to the foregoing discussions, we pick a tuple of four rational numbers in $[0,1)$,
	\begin{equation*}
	\alpha=(\alpha_0,\alpha_1,\alpha_\lambda,\alpha_\infty),
	\end{equation*} 
	which satisfies the equality
	\begin{equation*}
	\alpha_0+\alpha_1+\alpha_\lambda+\alpha_\infty=\frac{1}{2}+n, \quad 0\leq n\leq 3
	\end{equation*}
	Let 
	\begin{equation*}
	(E_{\alpha},\theta)=(E,\theta)\otimes (\sO(-n)(\alpha_00+\alpha_11+\alpha_\lambda \lambda+\alpha_\infty\infty),0)
	\end{equation*}
	be the corresponding parabolic Higgs bundle. Let $M_{\Higgs}$ be the moduli space of rank two and degree zero semistable graded logarithmic Higgs bundles over $(\P^1,D)/k$ containing $(E,\theta)$. It is naturally identified with the moduli space $M_{\alpha}$ of semistable graded parabolic Higgs bundles containing $(E_{\alpha},\theta)$. 
	
	\begin{proposition}\label{identification of self-maps}
	Notation as above. Suppose that each component of $\alpha$ is contained in $\frac{1}{N}\Z$, with $N$ coprime to $\Char(k)=p>0$. Then the flow operator $\Gr\circ C^{-1}$ induces a map
	\begin{equation*}
	\varphi_{\alpha}: M_{\alpha}(k)\to M_{\langle p\alpha\rangle}(k),
	\end{equation*}
	where $\langle p\alpha\rangle=(\langle p\alpha_0\rangle,\langle p\alpha_1\rangle,\langle p\alpha_\lambda\rangle,\langle p\alpha_\infty\rangle)$. Under the natural identifications
	\begin{equation*}
	M_{\alpha}\cong M_{\Higgs}\cong M_{\langle p\alpha\rangle},
	\end{equation*}
	$\varphi_{\alpha}$ is identified with the self map $\varphi_{\lambda,p}$ of Sun-Yang-Zuo.
	\end{proposition}
	\begin{proof}
	This is a special case of Lemma \ref{tensor product lemma}.
	\end{proof}

	There is one particularly simple choice of $\alpha$, namely the case when only one nonzero component of $\alpha$ is $\frac{1}{2}$. We shall simply write the moduli space by $M_{\frac{1}{2}P}$ when the component of $\alpha$ at $P$ is $\frac{1}{2}$. Note that for $p$ odd, the flow operator induces indeed a self map on $M_{\frac{1}{2}P}$ without resorting to the identification. Let $(E_{\frac{1}{2}P},\theta_x)$ represent a point in $M_{\frac{1}{2}P}(k)$. Then the isomorphism in Lemma \ref{tensor product lemma} yields a parabolic interpretation of the isomorphism \eqref{twist by O(m)} as follows:
	\begin{equation*}
	\begin{array}{rcl}
			\Gr\circ C^{-1}(E,\theta_x)&\cong & \Gr\circ C^{-1}(E_{\frac{1}{2}P},\theta_x)\otimes \Gr\circ C^{-1}(\sO(-\frac{1}{2}P),0)\\
			&\cong&(E_{\frac{1}{2}P},\theta_{x'})\otimes (\sO(-\frac{p}{2}P),0)\\
			&\cong&(E,\theta_{x'})\otimes (\sO(\frac{1}{2}P),0)\otimes (\sO(-\frac{p}{2}P),0)\\
			&\cong &(E,\theta_{x'})\otimes (\sO(\frac{1-p}{2}P),0).
		\end{array}
	\end{equation*}
	The moduli space $M_{\frac{1}{2}P}$ plays an important role in later sections.	
	The next three sections constitute the core of our method approaching to Conjecture \ref{SYZ conj}.
	Functoriality is the guiding principle of our treatment.

	\section{Elliptic curve as moduli space}
	In this section and next, let $k$ be an algebraically closed field with $\Char(k)\neq 2$, $C=C_{\lambda}$ be the elliptic curve over $k$, as defined in \S1, which is the double cover of $\P^{1}$ branched along $D=0+1+\lambda+\infty$, together with the projection map $\pi: C \to \P^1$ over $k$. Let $B$ be the reduced Weil divisor of $C$, formed by the pre-image of $|D|$ (so $\pi^*D=2B$). 
	We also write $\infty\in |B|$, the pre-image of $\infty\in |D|$, and set it as the origin of $C$. 
	Then we see at once that the elements in $ |B| $ are exactly the two-torsion points. Let $\sigma: C\to C$ be the canonical involution, and $G=\{\id, \sigma\}$ the constant group scheme over $k$ acting on $C$. One has $ \pi = \pi \circ \sigma$, and four elements of $|B| \subset C(k)$ are the fixed points of $\sigma$.

	The goal of this section is to construct a moduli space of logarithmic Higgs bundles over $(C,B)$, which is naturally isomorphic to $C$. Let us consider the moduli space $M_{\textrm{gr}}$ (resp. $M^s_{\textrm{gr}}$) of semistable (resp. stable) logarithmic Higgs bundles $(F, \eta)$ over $(C,B)$ satisfying the following conditions:
	\begin{itemize}
		\item[(\romannumeral1)] $F$ is of rank $2$;
		\item[(\romannumeral2)] $(F,\eta)=(F^{1,0}\oplus F^{0,1},\eta^{1,0}\oplus \eta^{0,1})$ is graded;
		\item[(\romannumeral3)] the determinant of $F$ is trivial.
	\end{itemize}
	\begin{lemma}
		$M_{\textrm{gr}}$ is a disjoint union of four sub-moduli spaces:
		\begin{equation*}
			M_{\textrm{gr}} = M_{0} \sqcup M_{1} \sqcup M_{2},
		\end{equation*}
		where 
  $M_{0}$ parameterizes $(L\oplus L^{-1},0)$ for $\deg L=0$, and  $M_{i}, i\geq 1$, parametrizes those of form $(L\oplus L^{-1},\eta)$ with $\deg L = i$ and $\eta \neq 0$. Consequently, 
  $$
  M^s_{\textrm{gr}}=M_{1} \sqcup M_{2}.
  $$
  \end{lemma}
	\begin{proof}
   Consider first the case that $\eta=0$. Hence $F$ is a rank two semistable vector bundle with trivial determinant over $C$. It is either decomposable or indecomposable. Taking the trivial determinant into account, one obtains
	$F\cong L\oplus L^{-1}$, and $\deg L=0$ as $F$ is semistable. In the latter case, $F$ must be a nontrivial extension of an $L$ by itself, and $L$ must be of two torsion because of the trivial determinant. But they are $S$-equivalent to elements in the former case.
	
 When $\eta\neq 0$, $F$ has to be a direct sum of two line bundles $L \oplus L^{-1}$, equipped with a nonzero morphism
	\begin{equation*}
		\eta^{1,0}:L \to L^{-1} \otimes \Omega_{C}(B).
	\end{equation*}
	Therefore $\deg (L) \leq 2$. Owning to the semistability, one has $\deg L \geq 0$. 
	Hence we obtain a decomposition into three disjoint union of submoduli spaces $M_i$ for $ 0 \leq i \leq 2$ according to the degree of $L$. It is clear that only when $\eta\neq 0$ and $\deg i\geq 1$, $(F,\eta)$ is stable. When $\deg L=0$, it lies in $M_0$ by the $S$-equivalence. The lemma follows.
	\end{proof}

   We are looking for some particular submoduli space $\Sigma_{\Higgs}$ in $M_{\textrm{gr}}$ isomorphic to $C$. The reason that we have to consider the whole moduli $M_{\textrm{gr}}$ is simply because a priori, in the positive characteristic situation, the flow operator maps $\Sigma_{\Higgs}$ to $M_{\textrm{gr}}$ (see \S5). 

Consider the geometry of $M^s_{\textrm{gr}}$: If $ \deg L = 2 $, the Higgs field is maximal. So one obtains an isomorphism 
	\begin{equation*}
		L^{\otimes 2}\cong \Omega_C(B),
	\end{equation*}
	from which it follows that $L\cong \sO_C(P+\infty)$ for some $P \in  |B|$. That means $M_2$ consists of disjoint union of four points. If $ \deg L=1$, by sending the pair $[(L\oplus L^{-1},\eta)]$ to $[L]$, one obtains a natural morphism 
	\begin{equation*}
		\pi_1: M_1\to \mathrm{Pic}^{1}(C) \cong C.
	\end{equation*}
	It is a ruled surface. Indeed, the fiber of $\pi_1$ at the point $[L]$ is naturally isomorphic to the complete linear system $|L^{-2}\otimes \Omega_C(B)|$, which is a $ \mathbb{P}^{1} $ by Riemann-Roch. In junction with $M_{\mathrm{Higgs}}$, we observe that there are natural morphisms $M_{\mathrm{Higgs}}\to M_1$ as follows: 
	Choose and fix a point $P\in |D|$. In Proposition \ref{identification of self-maps} we have seen $M_{\Higgs}\cong M_{\frac{1}{2}P}$. 
	\begin{lemma}\label{four fibers}
	Let $Q\in |B|$ be the pre-image of $P$.	Then the parabolic pullback induces a morphism of $k$-varieties
	\begin{equation*}
	M_{\frac{1}{2}P}\to M_1, 
	\end{equation*}
	which induces an isomorphism onto the closed fiber $\pi_1^{-1}(Q)$ of $\pi_1$. 
	\end{lemma}
	\begin{proof}
	By definition, the closed immersion $\pi_1^{-1}(Q)\to M_1$ is the submoduli space consisting of those logarithmic Higgs bundles $(F_Q=\mathcal{O}_{C}(Q)\oplus \mathcal{O}_{C}(-Q),\eta)$. 
	We claim that the functor $\pi_{\Par}^*$ induces an isomorphism of moduli spaces $M_{\frac{1}{2}P}\cong \pi_1^{-1}(Q)$. 
	Let $(E_{\frac{1}{2}P},\theta_x)$ represent an element in $M_{\frac{1}{2}P}(k)$. In the following, we show that $\pi_{\Par}^*(E_{\frac{1}{2}P},\theta_x)$ is \emph{naturally} isomorphic to $(F_Q,\eta_{\pi^{*}x})$, where the divisor of zero of the nonzero morphism 
	\begin{equation*}
	\eta_{\pi^{*}x}^{1,0}: \sO_C(Q)\to \sO_C(-Q)\otimes \Omega_{C}(B),
	\end{equation*}
	equals $\pi^*x$ (regarding $x$ as a Weil divisor). As $Q$ is a fixed point of $\sigma$, there is a natural isomorphism
	\begin{equation*}
	\sigma^*\sO_C(Q)\cong \sO_{C}(\sigma(Q))=\sO_C(Q).
	\end{equation*}
	So does its dual $\sO_C(-Q)$. Taking the direct sum of these isomorphisms, we obtain the isomorphism $\psi_{\sigma}: \sigma^*F_Q\cong F_Q$. Then $\psi=\{\psi_{\id}=\textrm{Id},\psi_{\sigma}\}$ defines a natural $G$-action on $F_Q$. Because 
	\begin{equation*}
	\sigma^*(\eta_{\pi^{*}x})=\eta_{\sigma^*\pi^*x}=\eta_{\pi^*x},
	\end{equation*}
	$\psi$ actually gives rise to a $G$-action on the Higgs bundle $(F_Q,\eta_{\pi^{*}x})$. Then by the BIS correspondence, it suffices to show the following natural isomorphism
	\begin{equation*}
	\pi_{\Par*}(F_Q,\eta_{\pi^{*}x})\cong (E_{\frac{1}{2}P},\theta_x).
	\end{equation*}
	But this is easy: Note that the $G$-action on the vector bundle $F_Q$ decomposes. It follows that
	\begin{equation*}
	\pi_{\Par*}(F_Q)\cong \pi_{\Par*}(\sO_C(Q))\oplus \pi_{\Par*}(\sO_C(-Q))\cong \sO_{\P^1}(\frac{1}{2}P)\oplus \sO_{\P^1}(-\frac{1}{2}P),
	\end{equation*}
	which is naturally isomorphic to $E_{\frac{1}{2}P}$. Clearly 
	\begin{equation*}
	\pi_{\Par*}(\eta_{\pi^{*}x}): \sO_{\P^1}(\frac{1}{2}P)\to \sO_{\P^1}(-\frac{1}{2}P)\otimes \Omega_{\P^1}(D)
	\end{equation*}
	is nonzero, and hence admits a simple zero (say $x'$) for degree reason. 
	It is easy to see that the support of the divisor of zero of $\pi_{\Par}^*(\theta_{x'})$ is contained in $\pi^{-1}x'$. Therefore, $x'=x$. 
	This fact can be also seen directly by a local computation on the parabolic pullback. 
	
	In order to obtain the full claim, we need to repeat the above argument for any $T$-valued points of the corresponding moduli functors to $M_{\frac{1}{2}P}$ and $\pi_1^{-1}(Q)$, where $T$ is any $k$-scheme of finite type. Since the BIS correspondence also works in the relative setting, we leave the necessary formal extension of the above argument as an exercise.  
	\end{proof}
	
	\begin{lemma}\label{four sections}
		The submoduli space of $M_1$ consisting of those Higgs fields vanishing only at one point (multiplicity not counted) is a disjoint union of four sections $\Sigma^Q_{\mathrm{Higgs}}$ of $\pi_1$, labelled by four two torsion points $Q\in |B|$. 
	\end{lemma}
	\begin{proof}
		Let $(\mathcal{O}_{C}(P)\oplus \mathcal{O}_{C}(-P), \eta)$ represent an element in $M_1$. The divisor of zero of $\eta^{1,0}$ is a closed subscheme $R+S$ of length two satisfying the equation
		\begin{equation*}
			[R+S]=[2\sigma(P)].
		\end{equation*}
		Here $ [ - + - ] $ denotes for the addition of the group scheme $C$. 
		Therefore, it lies in the submoduli modulie if and only if $R=S$ and $[2R]=[2\sigma(P)]$ hold. Clearly, the solutions are 
		\begin{equation*}
			R=S=[\sigma(P)+Q],
		\end{equation*}
		for a two torsion point $Q$. So the lemma follows.
	\end{proof}
	We denote by $\Sigma_{\mathrm{Higgs}}$ the section $\Sigma^Q_{\mathrm{Higgs}}$ where $Q$ is the origin of $C$. 
	Therefore, $\Sigma_{\mathrm{Higgs}}$ consists of rank two stable logarithmic Higgs bundles of the following form:
	\begin{equation}\label{Higgs bundle in canonical section}
		(\mathcal{O}_{C}(P)\oplus \mathcal{O}_{C}(-P), \eta_{2\sigma(P)}),
	\end{equation}
	where the Higgs field $\eta^{1,0}_{2\sigma(P)}$ vanishes at $\sigma(P)$ with order $2$. In the picture below, four black dots in $C$ represent two torsion points. 

 
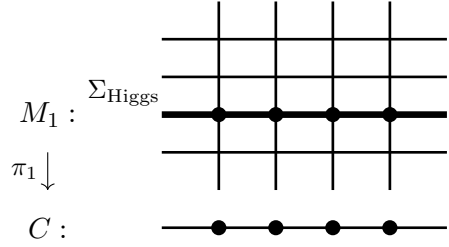
\begin{center}
		\begin{tikzpicture}

		\draw [line width=2.5pt](-1.5,0) -- (2.25,0);
		\node[above] at (-2,0) {$ \Sigma_{\Higgs} $};
		\draw [line width=1pt](-1.5,0.5) -- (2.25,0.5);
		\draw [line width=1pt](-1.5,-0.5) -- (2.25,-0.5);
		\draw [line width=1pt](-1.5,1) -- (2.25,1);
		\draw [line width=1pt](-1.5,-1.5) -- (2.25,-1.5);
		\draw [line width=1pt](0,1.5) -- (0,-1);
		\draw [line width=1pt](0.75,1.5) -- (0.75,-1);
		\draw [line width=1pt](-0.75,1.5) -- (-0.75,-1);
		\draw [line width=1pt](1.5,1.5) -- (1.5,-1);
		\fill (0,0) circle (.1);
		\fill (-0.75,0) circle (.1); 
		\fill (0.75,0) circle (.1); 
		\fill (1.5,0) circle (.1); 
		\fill (0,-1.5) circle (.1);
		\fill (-0.75,-1.5) circle (.1); 
		\fill (0.75,-1.5) circle (.1); 
		\fill (1.5,-1.5) circle (.1); 
		\node[font=\large](M1) at (-3,0){$ M_{1}: $};
		\node[font=\large](C) at (-3,-1.5){$ C: $};
		\node[left] at (-3,-0.75){$ \pi_{1} $};
		\draw[->] (-3,-0.5) -- (-3,-1);
		\end{tikzpicture}
	\end{center}

 \iffalse
	\begin{remark}
		This remark concerns the structure of the $\P^1$-bundle $\pi_1: M_1\to C$. 
		Regard $C\times C$ as a $C$-scheme via the second projection. 
		We may define a $C$-morphism 
		\begin{equation*}
			\alpha: C\times C\to M_1
		\end{equation*}
		by sending $(Q,P)$ to $(Q+[2\sigma(P)-Q],P)$. This is well-defined since  
		\begin{equation*}
			[Q+[2\sigma(P)-Q]]=[2\sigma(P)].
		\end{equation*}
		We have a second $C$-morphism $\beta: C\times C\to \P^1\times C$, defined by 
		\begin{equation*}
			(Q,P) \mapsto (\pi([P+Q]),Q).
		\end{equation*}
		One observes that $\alpha$ factors through $\beta$, so that one obtains a $C$-isomorphism
		$\gamma: \P^1\times C\to M_1$. 
		Therefore, $\pi_1$ is a trivial $\P^1$-bundle. The inverse of the trivialization may be described by
		\begin{equation*}
			\gamma^{-1}: M_{1} \rightarrow  \P^1 \times C, \quad (Q+[2\sigma(P)-Q], P) \mapsto (\pi([Q-\sigma(P)]), P). 
		\end{equation*}
		Under $\gamma^{-1}$, $\Sigma_{Higgs}$ is mapped isomorphically onto $\infty\times C$. 
	\end{remark}
	\fi

	\section{Covering map as natural transformation}
	Recall that $\pi: C\to \P^1$ is the covering map of degree two. In this section, we shall enhance it to a morphism of moduli spaces over $k$
	\begin{equation*}
	\pi_{\Higgs}: \Sigma_{\Higgs}\to M_{\Higgs}.
	\end{equation*}
	In other words, we shall construct a $k$-morphism $\pi_{\Higgs}$ which coincides with $\pi$ after the previous identifications $M_{\Higgs}\cong \P^1, \Sigma_{\Higgs}\cong C$. Before doing this, let us remark that for the sake of Conjecture \ref{SYZ conj}, it suffices to construct a set-theoretical map 
	\begin{equation*}
	\pi_{\Higgs}(k): \Sigma_{\Higgs}(k)\to M_{\Higgs}(k),
	\end{equation*}
	which coincides with $\pi(k)$ under the identifications. 
	The reason for doing so is to demonstrate power (and beauty) of the notion of a moduli functor, which turns the functorial property of an underlying construction into something more useful.

	Let $(\mathrm{Sch}/k)$ be the category of $k$-schemes of finite type. Let 
	\begin{equation*}
	\sM^s: (\mathrm{Sch}/k)^{\circ} \to (\mathrm{Sets})
	\end{equation*}
	be the moduli functor of logarithmic stable Higgs sheaves over $(C,B)/k$ of rank two and degree zero. 
	To be precise, for any $T\in (\mathrm{Sch}/k)$, $\sM^s(T)$ is the set of equivalence classes of logarithmic Higgs modules $(\sF,\eta)$ over $(C_T,B_T)$, where $\sF$
	is a coherent $\sO_{C_T}$-module of rank two, 
	\begin{equation*}
	\eta: \sF\to \sF\otimes_{\sO_{C_T}}\Omega_{C_T/T}(\log B_T)
	\end{equation*}
	is $\sO_{C_T}$-linear satisfying $\eta\wedge \eta=0$, and for each geometric point $t\in T$, the fiber $(F_t,\eta_t):=(\sF,\eta)|_{C_t}$ is a rank two stable logarithmic Higgs bundle over $(C_t,B_t)/k$ of degree zero. 
	In above, $(\sF,\eta)\equiv(\sF',\eta')$, if there is an invertible sheaf $L$ over $T$ such that there is an isomorphism of Higgs sheaves:
	\begin{equation*}
	(\sF,\eta)\cong(\sF',\eta')\otimes (p_T^*L,0),
	\end{equation*}
	where $p_T: C_T\to T$ is the projection. 
	It is a well-established theorem that the moduli functor $\sM$ is \emph{universally corepresented} by a quasi-projective scheme $M^s$ over $k$ (see \cite[Theorem 4.3.4, 3.B]{HL10}, \cite{Ma16}, \cite{Sim94}, \cite{Lan14}). 
	Therefore, for each morphism $T\to M^s$ in $(\mathrm{Sch}/k)$, the fibre product $\sT:=\underline{T}\times_{\underline M^s}\sM^s$ is (universally) copresented, where $\alpha: \sM^s\to \underline M^s$ is constructed by the moduli problem, and $\underline T: (\mathrm{Sch}/k)^{\circ} \to (\mathrm{Sets})$ be the functor represented by $T$.  
	
	The universal corepresentability of $\sM^s$ allows us to define various submoduli functors for logarithmic Higgs bundles over $(C,B)/k$ with additional structures. 
	Let $\Sigma_{\Higgs}\to M^s$ be the closed immersion, whose image in $M^s(k)$ is the closed subset consisting of the isomorphism classes of form \eqref{Higgs bundle in canonical section}.
	\begin{lemma}\label{Description of Sigma}
	The functor $\varSigma_{\Higgs}=\underline{\Sigma_{\Higgs}}\times_{\underline M^s}\sM^s$ is copresented by $\Sigma_{\Higgs}$. 
	For $T\in (\mathrm{Sch}/k)$, $\varSigma_{\Higgs}(T)$ is the set of equivalence classes of graded Higgs modules $(\sL\oplus \sL^{-1}\otimes p_T^*L,\eta^{1,0}\oplus \eta^{0,1}=0)$, where $\sL$ is an invertible sheaf over $C_T$, flat over $T$, $L$ is an invertible sheaf over $T$, and for each geometric point $t\in T$, the restriction of $(\sL\oplus \sL^{-1}\otimes p_T^*L,\eta)$ to $C_t$ represents an element in $\Sigma_{\Higgs}(k)$. 
	In above, the equivalence is the same as the one defined for $\sM^s$.
	\end{lemma}
	\begin{proof}
	By definition, an element in $\varSigma_{\Higgs}(T)$ is representable by a logarithmic Higgs module $(\sF,\eta)$ over $(C_T,B_T)$ whose associated moduli map $T\to M^s$ factors through the closed subscheme $\Sigma_{\Higgs}$. 
	All statements are clear except the existence of a graded structure on $(\sF,\eta)$ as claimed.
	
	Because the Harder-Narasimhan stratification of $M^s$ (ignoring the Higgs structure of a local universal family of stable Higgs bundles) is constant along $\Sigma_{\Higgs}$, the relative HN filtration exists for $\sF$ over $T$ by \cite[Theorem 5]{N11}. 
	It takes the form $\sL\subset \sF$, where $\sL$ is invertible and flat over $T$. Consider the composite $\eta^{1,0}$ of natural morphisms:
	\begin{equation*}
	\sL\to \sF\stackrel{\eta}{\to}\sF\otimes \Omega_{C_T/T}(\log B_T)\to \sF/\sL\otimes \Omega_{C_T/T}(\log B_T).
	\end{equation*}
	Set $(\sF',\eta')=(\sL\oplus \sF/\sL,\eta^{1,0}\oplus \eta^{0,1}=0)$, which also represents an element in $\varSigma_{\Higgs}(T)$. Since over any $k$-point of $T$, 
	\begin{equation*}
	[(\sF',\eta')|_{C_t}]=[(\sF,\eta)|_{C_t}],
	\end{equation*}
	it follows that over $T_{\textrm{red}}$, the associated moduli map $\phi: T\to \sigma_{\Higgs}$ to $(\sF,\eta)$ coincides with the one $\phi'$ to $(\sF,\eta)$. In other words, 
	\begin{equation*}
	(\sF,\eta)|_{T_{\textrm{red}}}\equiv (\sF',\eta')|_{T_{\textrm{red}}}.
	\end{equation*}
	Assume that $T$ is reduced. Since the restriction of $\det(\sF)$ to any $C_t$ is trivial, it follows that $\det(\sF)\cong p_T^*L$ for an invertible sheaf $L$ over $T$. Thus 
	\begin{equation*}
	\sF/\sL\cong \sL^{-1}\otimes \det(\sF)\cong \sL^{-1}\otimes p_T^*L.
	\end{equation*}
	This proves the lemma for a reduced $T$. 
	To extend it to the non-reduced case, we shall make more effective use of the existence of a local universal family for $\sM^s$ (\cite[Theorem 4.7]{Sim94}, \cite[Proposition 7.6, 8.2]{BR85}). 
	Note that the pullback of a local universal family of $\sM^s$ to $\varSigma_{\Higgs}$ is a local universal family for $\varSigma_{\Higgs}$. 
	Let $\{U_i\to \Sigma_{\Higgs}\}_{i\in I}$ be a covering of $\Sigma_{\Higgs}$ in \'etale topology. As $\Sigma_{\Higgs}$ is reduced, $U_i$ is reduced too. 
	So the above argument applies to a local universal family for $\varSigma_{\Higgs}$. 
	As a consequence, there exists an \'etale covering $\{U_i\}_i$ of $\Sigma_{\Higgs}$ and a local universal family over each $(C_{U_i},B_{U_i})/k$:
	\begin{equation*}
	(\sF^{univ}_i=\sL_i\oplus \sL_i^{-1},\eta^{univ}_i=\eta_i^{1,0}\oplus \eta_i^{0,1}).
	\end{equation*}
	(Note by possibly taking an additional degree two \'etale cover, we may assume $\det \sF$ to be trivial). 
	With it in hand, we obtain an \'etale covering $\{T_i:=T\times_{\Sigma_{\Higgs}}U_i\}_{i\in I}$ of $T$ such that 
	\begin{equation*}
	(\sF,\eta)|_{T_i}\cong \phi_i^{*}(\sF^{univ}_i,\eta^{univ}_i)\otimes (p_{T_i}^*L_i,0),
	\end{equation*}
	where $\phi_i=\phi|_{T_i}$ and $L_i$ is some invertible sheaf over $T_i$. 
	Making a further \'etale base change, we may assume $L_i$ to be trivial. 
	By abuse of notation, we denote the \'etale covering again by $\{T_i\}_i$. 
	Therefore over $T_i$, $\phi=\phi'$, and $\det(\sF)$ is trivial. 
	Then it follows that $\phi=\phi'$, and there exists an invertible sheaf $L$ over $T$(=$T_{\textrm{zar}}$) such that $\det(\sF)\cong p_T^*L$. 
	The lemma is proved.   
	\end{proof}
	\begin{remark}
	The above proof uses no other property of $\Sigma_{\Higgs}$ than that it lies in \emph{one} HN stratum in the stable locus $M^s$. 
	Therefore, the similar description in the lemma holds \eg also for the moduli functor attached to $M_1\to M^s$.
	\end{remark}
	The moduli space $M_{\Higgs}$ as studied by Sun-Yang-Zuo corepresents the moduli functor $\sM_{\Higgs}$. In this case, $\sM_{\Higgs}$ is just the whole moduli functor of rank two stable logarithmic Higgs bundles over $(\P^1,D)$ of degree $-1$. 
	Now we proceed to construct a natural transformation from $\varSigma_{\Higgs}$ to $\sM_{\Higgs}$.

	For $(\sF,\eta)\in \varSigma_{\Higgs}(T)$, we set
	 \begin{equation*}
			(\sF,\eta)[G]=(\sF,\eta) \oplus \sigma_T^{*}(\sF,\eta).
	\end{equation*}
	It is a rank four logarithmic Higgs bundle over $(C_T,B_T)/T$ which admits a natural $G$-structure $\phi=\{\textrm{Id},\phi_{\sigma}\}$ defined by 
	\begin{equation*}
	\phi_{\sigma}: \sigma_T^*((\sF,\eta)[G])\to (\sF,\eta)[G], \quad (a,b)\mapsto (\sigma_T^*b,\sigma_T^*a).
	\end{equation*}
	By Lemma \ref{Description of Sigma}, we may write $\sF=\sL\oplus \sL^{-1}\otimes p_T^*L$. Hence we have
	\begin{equation*}
	\sF[G]=\sL\oplus \sL^{-1}\otimes p_T^*L\oplus \sigma_T^*\sL\oplus \sigma_T^*\sL^{-1}\otimes p_T^*L=(\sL\oplus \sigma_T^*\sL)\bigoplus (\sL^{-1}\oplus \sigma_T^*\sL^{-1})\otimes p_T^*L.
	\end{equation*}
	The last decomposition into a direct sum of rank two factors is a decomposition of $G$-equivariant bundles. 
	Let $\Pic(C)$ be the $k$-scheme representing the Picard functor for $C/k$. Let $\Pic^1(C)$ be the connected component which parametrizes degree one line bundles on $C$. 
	Let $\triangle\subset C\times_kC$ be the diagonal. 
	Then the line bundle $\sP:=\sO_{C_C}(\triangle)$ defines a $C$-flat family of degree one line bundles over $C$. 
	The associated moduli map $C\to \Pic^1(C)$ is an isomorphism, which gives us the identification $\Pic^1(C)=C$. Fix a section $s\in \Gamma(C_C,\sP)$ whose divisor of zeros is $\triangle$. 
	Let $\phi_{\sL}: T\to C$ be the moduli map associated to $\sL$. 
	Possibly tensoring with the pullback of a line bundle over $T$, we may simply take $\sL=(\id\times \phi_{\sL})^*\sP$. 
	Then the section $s$ pulls back to a section 
	$s_{\sL}\in \Gamma(C_T,\sL)$, whose divisor of zeros is easily seen to be the transport of the closed subscheme $\Gamma_{\phi_{\sL}}\subset T\times_kC$, the graph of $\phi_{\sL}$. 
	By abuse of notation, we set $P_T\subset C_T$ to be the transport and write $\sL=\sO_{C_T}(P_T)$, to indicate that there is a tautological section $s_{\sL}$ of $\sL$ given as above. 
	So we write $\sL^{-1}=\sO_{C_T}(-P_T)$, $\sigma_T^*\sL=\sO_{C_T}(\sigma_T(P_T))$, and etc.. 
	\begin{lemma}
	Notation as above. Then there is a short exact sequence of $G$-equivariant bundles over $C_T$
	\begin{equation}\label{source bundle}
		0 \to \mathcal{O}_{C_T} \stackrel{\Phi_{1}}{\longrightarrow} \mathcal{O}_{C_T}(P_T) \oplus \mathcal{O}_{C_T}(\sigma_T (P_T)) \stackrel{\Phi_{2}}{\longrightarrow} \mathcal{O}_{C_T}(P_T+\sigma_T (P_T)) \to 0,
	\end{equation}
	 where the $G$-action on $\mathcal{O}_{C_T}$ is the trivial one, and the one on $\mathcal{O}_{C_T}(P_T+\sigma_T (P_T))$ is determined by sending the global section $s_{\sL}\otimes \sigma_T^*s_{\sL}$ to $-(s_{\sL}\otimes\sigma_T^*s_{\sL})$.
	\end{lemma}
	\begin{proof}
	$\Phi_1$ is the morphism sending 1 to $(s_{\sL},\sigma_T^*s_{\sL})$, which is invariant under the $G$-action. We have the morphism $\sigma_T^*s_{\sL}$ given by
	\begin{equation*}
	\mathcal{O}_{C_T}(P_T)=\mathcal{O}_{C_T}(P_T)\otimes \sO_{C_T}
	\stackrel{\id\otimes \sigma_T^*s_{\sL}}{\longrightarrow}  
	\mathcal{O}_{C_T}(P_T)\otimes \sO_{C_T}(\sigma_T (P_T))=\mathcal{O}_{C_T}(P_T+\sigma_T (P_T)),
	\end{equation*}
	and similarly the morphism
	\begin{equation*}
	s_{\sL}: \sO_{C_T}(\sigma_T(P_T)) \to  \mathcal{O}_{C_T}(P_T+\sigma_T (P_T)).
	\end{equation*}
	$\Phi_2$ sends $(a,b)$ to $\sigma_T^*s_{\sL}(a)-s_{\sL}(b)$.
	
	Clearly, $\Phi_1$ is injective and $\mathrm{im}(\Phi_1)\subset \ker(\Phi_2)$. 
	We observe that the sequence in the lemma is the pullback via $\id\times \phi_{\sL}$ of the sequence for the universal bundle $\sP$. 
	We call it the universal sequence. By Nakayama's lemma, it suffices to show the exactness for the universal sequence over any closed point of $C\times C$. 
	But then it suffices to show the exactness of the sequence \eqref{source bundle} when $T$ is a closed point of $C$. 
	In this case, it is obvious that $\Phi_2$ is surjective (by embedding all invertible sheaves into the constant sheaf $k(C)$ as $\sO_C$-submodules). 
	Then the sequence is exact for degree reason. 
	\end{proof} 
	Tensoring the exact sequence \eqref{source bundle} with $\mathcal{O}_{C_T}(-P_T-\sigma_T (P_T))$, which is endowed with the trivial $G$-action, we obtain the second exact sequence of $G$-equivariant bundles:
	\begin{equation}\label{target bundle}
		0 \to \mathcal{O}_{C_T}(-P_T-\sigma_T(P_T))  \stackrel{\Psi_{1}}{\longrightarrow} \mathcal{O}_{C_T}(-P_T) \oplus \mathcal{O}_{C_T}(-\sigma_T (P_T))\stackrel{\Psi_{2}}{\longrightarrow} \mathcal{O}_{C_T}  \to 0.
	\end{equation}
	We remind that $\sO_{C_T}$ in the quotient is equipped with the $G$-action determined by sending the global section 1 to -1.
		
	Set 
	\begin{equation*}
	\sS=\sS^{1,0}\oplus \sS^{0,1}=\sO_{C_T}\oplus \mathcal{O}_{C_T}(-P_T-\sigma_T(P_T))\otimes p_T^*L\subset \sF[G].
	\end{equation*}
	\begin{proposition}\label{ses of E[G]}
	Notation as above. Then $\eta[G]$ preserves $\sS$. More precisely, the Higgs field 
	\begin{equation*}
	\eta[G]^{1,0}=\eta^{1,0}\oplus \sigma_T^*\eta^{1,0}: \sL\oplus \sigma_T^*\sL\to (\sL^{-1}\oplus \sigma_T^*\sL^{-1})\otimes p_T^*L\otimes \Omega_{C_T/T}(\log B_T) 
	\end{equation*}
	sends $\sS^{1,0}$ into $\sS^{0,1}\otimes \Omega_{C_T/T}(\log B_T)$. 
	Consequently, there is an exact sequence of $G$-equivariant logarithmic Higgs bundles over $(C_T,B_T)$:
	\begin{equation}\label{s.e.s. of E[G]}
	0 \to (\sS,\eta_{\textrm{sub}}) \to (\sF,\eta)[G] \to (\sQ,\eta_{\textrm{quo}}) \to 0,
	\end{equation}
	with $(\sQ,\eta_{\textrm{quo}})\cong (\sS,\eta_{\textrm{sub}}) \otimes (\mathcal{O}_{C_T}(P_T+\sigma_T (P_T)),0)$, where $\mathcal{O}_{C_T}(P_T+\sigma_T (P_T))$ is equipped with the $G$-action as given in Lemma \ref{source bundle}. 
	Furthermore, for each closed point $t\in T$, the restriction of $(\sS,\eta_{\textrm{sub}})$ to $C_t$ is isomorphic (as graded logarithmic Higgs bundle) to $(\sO_{C}\oplus \sO_C(-P-\sigma(P)),\eta_{P+\sigma(P)})$ for a unique $P\in C(k)$ and $\eta_{P+\sigma(P)}^{1.0}: \sO_{C}\to \sO_C(-P-\sigma(P))\otimes \Omega_{C}(B)$ is the unique morphism (up to scalar) whose divisor of zeroes is $P+\sigma(P)$. 
	\end{proposition}
	\begin{proof}
	Combining the exact sequences \eqref{source bundle} and \eqref{target bundle} together, we obtain the exact sequence \eqref{s.e.s. of E[G]} without Higgs field. 
	The upshot of the proof is to exhibit the structure of the Higgs field $\eta[G]^{1,0}$ with respect to $\sS^{1,0}$. 
	
	The problem is \'etale local, and it suffices to argue it for a local universal family $(\sF_i^{univ},\eta_i^{univ})$ (see the proof of Lemma \ref{Description of Sigma}). 
	In particular, we may assume $T=\Spec(R)$ with $R$ an integral domain, and that $L$ is trivial. Let $U=(C-\pi^{-1}\{0,\infty\})\times T$ and $U'=(C-\pi^{-1}\{1,\lambda\})\times T$. 
	They form an open affine cover of $C_T$. 
	In below, we shall carry out our computations over $U$ and simply state the parallel result over $U'$. 
	We shall also use the affine presentation of $C_T$, namely, $C_T\subset \A^2_T=\Spec(R[x,y])$ defined by the equation $y^2=x(x-1)(x-\lambda)$. 
	Then the closed subscheme $P_T\subset C_T$ is assigned with a 'coordinate' $(x_{P_T},y_{P_T})\in R^2$, where $x_{P_T}$ (resp. $y_{P_T}$) is the image of $x$ (resp. $y$) under the morphism $R[x,y]\to R$ given by $\phi_{\sL}\times \id: T\to C\times T=C_T$. 
	For a local universal family, $y_{P_T}$ is not the zero in $R$. 
	Shrinking $T$ suitably, we may even assume $y_{P_T}$ is a unit of $R$. 
	Since the $\sigma_T$-action on $R[x,y]$ maps $x$ to $x$, and $y$ to $-y$, the 'coordinate' of $\sigma_T(P_T)$ is $(x_{P_T},-y_{P_T})$.
	Also, we may assume that $\sO_{C_T}(P_T+\sigma_T(P_T))$ can be trivialized over $U$:
	\begin{equation*}
	\sO_{C_T}(P_T+\sigma_T(P_T))(U) = b^{+}\sO_{C_T}(U),\ \sO_{C_T}(-P_T-\sigma_T(P_T))(U) = b^{-}\sO_{C_T}(U).
	\end{equation*}
	Set $e_{1}=\Phi_{1}(1_{U})$ and $f_{1}=\Psi_{1}(b^{-})$, and take (to be determined) liftings $e_{2}$ (resp. $f_{2}$) of $b^{+}$ (resp. $1_{U}$) with respect to $\Phi_{2}$ (resp. $\Psi_{2}$). 
	Then $\{e_1,e_2\}$ (resp. $(f_1,f_2)$) forms a local basis of $\sF[G]^{1,0}$ (resp. $\sF[G]^{0,1}$). 
	We also need to fix a local basis $d\log x'$ of $\Omega_{C_T/T}(\log B_T)$ over $U$, where $x'=\frac{x-\lambda}{x-1}$. 
			
	As $C_T$ is integral by our assumption, we may embed all invertible sheaves into the constant sheaf $k(C_T)$. 
	Using these embeddings, $\Phi_1$ (resp. $\Phi_2$) becomes the addition (resp. subtraction) of natural inclusions. 
	Set
	\begin{equation*}
	b_{P_T}=\frac{y+y_{P_T}}{x-x_{P_T}},\ b_{\sigma_{T}(P_T)}=\frac{y-y_{P_T}}{x-x_{P_T}}, \ b_{P_T+\sigma_{T}(P_T)}=\frac{1}{x-x_{P_T}}.
	\end{equation*}
	Note that over $U$, they have the only simple pole at $P_T$, $\sigma_{T}(P_T)$, and $P_T+\sigma_{T}(P_T)$ respectively. 
	We calculate that 
	\begin{equation*}
	\Phi_1(1_U)=1_U+1_U=\frac{x-x_{P_T}}{y+y_{P_T}}b_{P_T}+\frac{x-x_{P_T}}{y-y_{P_T}}b_{\sigma_{T}(P_T)},
	\end{equation*}
	and 
	\begin{equation*}
	\Phi_2(\frac{1}{2y_T}(b_{P_T}+b_{\sigma_{T}(P_T)}))=\frac{1}{2y_T}b_{P_T}-\frac{1}{2y_T}b_{\sigma_{T}(P_T)}=b_{P_T+\sigma_{T}(P_T)}.
	\end{equation*}
	Note that $b_{P_T+\sigma_{T}(P_T)}$ is actually a basis for $\sO_{C_T}(P_T+\sigma_T(P_T))(U)$. 
	As $2y_T$ is a unit, we may take $e_2=\frac{1}{2y_T}(b_{P_T}+b_{\sigma_{T}(P_T)})$. 
	Let $V\subset C_T$ be a smaller open subset such that $b_{P_T}$ (resp. $b_{\sigma_{T}(P_T)}$) is a basis of $\sO_{C_T}(P_T)(V)$ (resp. $\sO_{C_T}(\sigma_T(P_T))(V)$. 
	Then we obtain the following transformation matrix of local bases for $(\sO_{C_T}(P_T)\oplus \sO_{C_T}(\sigma_T(P_T)))(V)$:
	\begin{equation}\label{basis transformation}
	\begin{bmatrix}
		e_{1} & e_{2}
	\end{bmatrix} = 
	\begin{bmatrix}
		b_{P_T} & b_{\sigma_T(P_T)}
	\end{bmatrix} \cdot 
	\begin{bmatrix}
		\frac{x-x_{P_T}}{y+y_{P_T}} & \frac{1}{2 y_{P_T}} \\
		\frac{x-x_{P_T}}{y-y_{P_T}} & \frac{1}{2 y_{P_T}}
	\end{bmatrix}(:=A).
	\end{equation}
	Set $\check{b}_{P_T}=b_{P_T}^{-1}$ \etc, and repeat the above argument for the sequence \eqref{target bundle}. We obtain a local lifting $f_2$ of $1_U\in \sO_{C_T}(U)$, and the following transformation matrix of local bases for $(\sO_{C_T}(-P_T)\oplus \sO_{C_T}(-\sigma_T(P_T)))(V)$:
	\begin{equation}\label{second basis transformation}
	\begin{bmatrix}
		f_{1} & f_{2}
	\end{bmatrix} = 
	\begin{bmatrix}
		\check{b}_{P_T} & \check{b}_{\sigma_T(P_T)}
	\end{bmatrix} \cdot 
	\begin{bmatrix}
			(y+y_{P_T}) & \frac{y^{2}-y^{2}_{P_T}}{-2y_{P_T}(x-x_{P_T})} \\
			(y-y_{P_T}) & \frac{y^{2}-y^{2}_{P_T}}{-2y_{P_T}(x-x_{P_T})}
		\end{bmatrix}(:=B).
	\end{equation}
	The representation matrix of $\eta[G]^{1,0}$ with respect to the local bases $\{b_{P_T},\check{b}_{P_T}\}$  (resp. $\{b_{\sigma_T(P_T)},\check{b}_{\sigma_T(P_T)}\}$) of $(\sO_{C_T}(P_T)\oplus \sO_{C_T}(\sigma_T(P_T)))$ (resp. $(\sO_{C_T}(-P_T)\oplus \sO_{C_T}(-\sigma_T(P_T)))$) is by definition given by a matrix of the following form
	\begin{equation*}
		\Theta_{\textrm{def}}=
		\begin{bmatrix}
			(y+y_{P_T})^{2} & 0 \\
			0 & (y-y_{P_T})^{2}
		\end{bmatrix} \otimes u\cdot d\log x',
	\end{equation*}
	where $u$ is some local unit. 
	Now we write
	\begin{equation*}
		\eta[G]|_{U} \begin{bmatrix}
			e_{1} & e_{2}
		\end{bmatrix}=\begin{bmatrix}
			f_{1} & f_{2}
		\end{bmatrix} \cdot \Theta.
	\end{equation*}
	The above calculations lead to $\Theta=B^{-1}\cdot\Theta_{\textrm{def}}\cdot A$, which takes the following form:
	\begin{equation*}
		\Theta=
		\begin{bmatrix}
			(x-x_{P_T}) & \frac{y}{y_{P_T}} \\
			0 & (x-x_{P_T})
		\end{bmatrix}  \otimes u\cdot \dif \log x',
	\end{equation*}
	Over the other open subset $U'$, the representation matrix $\Theta'$ of $\eta[G]|_{U'}$ with respect to analogous bases is of form  
	\begin{equation*}
		\Theta'=
		\begin{bmatrix}
			(x'-x'_{P_T}) & \frac{y'}{y'_{P_T}} \\
			0 & (x'-x'_{P_T})
		\end{bmatrix}   \otimes u'\cdot d\log x,
	\end{equation*}
	where $x'_{P_T}=\frac{x_{P_T}-\lambda}{x_{P_T}-1}$, $y'=\frac{(\lambda-1)y}{(x-1)^{2}}$ and $u'$ is a unit. 
	Now the statement on the Higgs structure becomes more than obvious.
	\end{proof}
	\begin{remark}
	The local units $u,u'$ can be also determined (up to units of $R$), using the compatibility of the two representation matrices. In fact, 
	\begin{equation*}
			u=(x(x_{P_T}-1))^{-1}, \quad u'=-(x'(x'_{P_T}-1))^{-1}.
			\end{equation*}
	They are useful in an explicit calculation of the inverse Cartier transform of $(\sF,\eta)[G]$ in positive characteristic. 		 
	\end{remark}
	Let $\pi_T: C_T\to \P^1_T$ be the natural projection. 
	Set $(\sG,\Theta)=\pi_{T\Par*}(\sF,\eta)[G]$, the parabolic pushforward of $(\sF,\eta)[G]$. 
	It is a $T$-flat family of rank four polystable parabolic Higgs bundles over $(\P^1,D)$ of degree zero. 
	Ignoring its parabolic structure, we obtain a $T$-flat family of rank four logarithmic Higgs bundles over $(\P^1,D)$ viz. $(\sG_0,\Theta)$.  
	\begin{corollary}\label{max destabilizer}
	Notation as above. 
	Then the relative Harder-Narasimhan filtration for $(\sG_0,\Theta)$ exists. 
	Let $(\sE,\Theta)$ be its maximal destabilizer. Then it represents an element in $\sM_{\Higgs}(T)$.
	\end{corollary}
	\begin{proof}
	The functor $\pi_{T\Par*}$ is exact. 
	So we obtain from \eqref{s.e.s. of E[G]} the following exact sequence of parabolic Higgs bundles over $(\Pl_T,D_{T})/T$:
	\begin{equation*} \label{ses over P1}
		0 \to \pi_{T\Par*}(\sS,\eta_{\textrm{sub}}) \to  (\sG,\Theta) \to \pi_{T\Par*}(\sQ,\eta_{\textrm{quo}}) \to 0.
	\end{equation*}
	We are going to show that $(\sE,\Theta)=\pi_{T\Par*}(\sS,\eta_{\textrm{sub}})$ as logarithmic Higgs bundle over $(\P^1_T,D_T)/T$, and it represents an element in $\sM_{\Higgs}(T)$. 
	
	First note that $\sS$ is a direct sum of two $G$-equivariant line bundles with trivial $G$-structure. 
	Let $q_T: \P^1_T\to T$ be the natural projection. 
	Then 
	\begin{equation*}
	\pi_{T\Par*}(p_T^*L)=\pi_{T\Par*}(\pi_T^*q_T^*L)=\pi_{T\Par*}(\pi_{T\Par}^*q_T^*L)=q_T^*L.
	\end{equation*}
	It follows that
	\begin{equation*}
	\pi_{T\Par*}(\sS)=\pi_{T\Par*}(\sO_{C_T})\oplus \pi_{T\Par*}(\sO_{C_T}(-P_T-\sigma_T(P_T))\otimes q_T^*L,
	\end{equation*}
	which has the trivial parabolic structure. Let $t$ be a closed poin of $T$. By Proposition \ref{ses of E[G]}, it follows that
	\begin{equation*}
	\pi_{T\Par*}(\sS,\eta_{\textrm{sub}})|_{\P^1_t}\cong \pi_{\Par*}((\sS,\eta_{\textrm{sub}})|_{C_t})\cong \pi_{\Par*}(\sO_{C}\oplus \sO_C(-P-\sigma(P)),\eta_{P+\sigma(P)}),
	\end{equation*}
	for some unique closed point $P\in C(k)$. 
	Let $Q=\pi(P)\in \P^1(k)$. 
	As
	\begin{equation*}
	\pi_{\Par}^*\sO_{\P^1}(-Q)=\pi^*\sO_{\P^1}(-Q)=\sO_C(-P-\sigma(P))
	\end{equation*}
	with the trivial $G$-structure, 
	\begin{equation*}
	\pi_{\Par*}(\sO_{C}\oplus \sO_C(-P-\sigma(P))=\sO_{\P^1}\oplus \sO_{\P^1}(-Q)\cong \sO_{\P^1}\oplus \sO_{\P^1}(-1).
	\end{equation*}
	Moreover, the Higgs field $\pi_{\Par*}\eta_{P+\sigma(P)}$ is easily seen to have a unique simple zero at $Q$. 
	Therefore, $\pi_{T\Par*}(\sS,\eta_{\textrm{sub}})$ represents an element in $\sM_{\Higgs}(T)$. 
	
	On the other hand, by Proposition \ref{ses of E[G]} again, we compute that
	\begin{equation*}
	\begin{array}{rcl}
		\pi_{T\Par*}(\sQ,\eta_{\textrm{quo}})|_{\P^1_t}&\cong &\pi_{T\Par*}((\sS,\eta_{\textrm{sub}})\otimes (\mathcal{O}_{C_T}(P_T+\sigma_T (P_T)),0))|_{\P^1_t}\\
		&\cong&(\sO_{\P^1}\oplus \sO_{\P^1}(-1), \theta_Q)\otimes  (\pi_{\Par*}\mathcal{O}_{C}(P+\sigma (P)),0)\\
		&\cong&(\sO_{\P^1}\oplus \sO_{\P^1}(-1),\theta_Q)\otimes (\pi_{\Par*}(\sO_{C}\otimes \pi_{\Par}^*\sO_{\P^1}(Q)),0)\\
    &\cong &(\sO_{\P^1}\oplus \sO_{\P^1}(-1),\theta_Q)\otimes (\sO_{\P^1}(Q)\otimes \pi_{\Par*}\sO_{C},0) \\
		&\cong &((\sO_{\P^1}(-1)\oplus \sO_{\P^1}(-2))(\frac{1}{2}D),\theta_Q).
	\end{array}
	\end{equation*}
	Therefore, we have the following exact sequence of logarithmic Higgs bundles over $(\P^1,D)/k$:
	\begin{equation*}
		0 \to (\OO_{\Pl} \oplus \OO_{\Pl}(-1),\theta_Q) \to (\sG_0,\Theta)|_{\P^1_t} \to (\OO_{\Pl}\oplus \OO_{\Pl}(-1),\theta_{Q})\otimes (\sO_{\P^1}(-1),0)\to 0.
	\end{equation*}
	Clearly, the HN polygon of the $T$-family $(\sG_0,\Theta)$ keeps constant. 
	So it admits a unique relative HN filtration, which is nothing but $\pi_{T\Par*}(\sS,\eta_{\textrm{sub}})\subset (\sG_0,\Theta)$.
	\end{proof}
	For $T\in (\mathrm{Sch}/k)$, we have a well-defined map of sets
	\begin{equation*}
	\Pi_{\Higgs}(T): \varSigma_{\Higgs}(T)\to \sM_{\Higgs}(T),
	\end{equation*}
	which sends the equivalence class of $(\sF,\eta)$ to the equivalence class of the maximal destabilizer of $(\pi_{T\Par*}(\sF,\eta)[G])_0$. 
	For a morphism of $k$-schemes $f: T\to T'$, we leave the reader to verify the commutativity of the following diagram:
	\begin{equation*}
	\xymatrix{ \varSigma_{\Higgs}(T)\ar[r]^-{\Pi_{\Higgs}(T)}\ar[d]_{\varSigma_{\Higgs}(f)}&\sM_{\Higgs}(T)\ar[d]^{\sM_{\Higgs}(f)}\\
					\varSigma_{\Higgs}(T')\ar[r]^-{\Pi_{\Higgs}(T')}&\sM_{\Higgs}(T').}
	\end{equation*}	
	Summarizing the above discussions, we finally obtain the main result of the section.	 
	\begin{proposition}\label{moduli interpretation of pi}
	Notation as above. 
	The natural transformation $\Pi_{\Higgs}$ from the moduli functor $\varSigma_{\Higgs}$ to the moduli functor $\sM_{\Higgs}$ induces a $k$-morphism 
	\begin{equation*}
	\pi_{\Higgs}: \Sigma_{\Higgs}\to M_{\Higgs}.
	\end{equation*} 
	Under the natural identifications $\Sigma_{\Higgs}\cong C$ and $M_{\Higgs}\cong \P^1$, $\pi_{\Higgs}$ coincides with $\pi$.
	\end{proposition}
	\begin{proof}
	By the corepresentability of moduli functors, there exists a unique morphism $\pi_{\Higgs}:\Sigma_{\Higgs}\to M_{\Higgs}$ of $k$-schemes rendering the commutativity of the following diagram of natural transformations: 
	\begin{equation*}
	\xymatrix{ \varSigma_{\Higgs}\ar[r]^-{\Pi_{\Higgs}}\ar[d]_{\alpha}&\sM_{\Higgs}\ar[d]^{\alpha}\\
					\underline{\Sigma_{\Higgs}}\ar[r]^-{\underline{\pi_{\Higgs}}}&\underline{M_{\Higgs}}.}
	\end{equation*}
	(Underline means the Yoneda embedding). 
	Under the natural identifications, the proof of Corollary \ref{max destabilizer} shows that on the set of $k$-points, $\pi_{\Higgs}$ coincides with $\pi$. 
	Hence they coincide as morphism of $k$-schemes too.  
	\end{proof}

\section{Flow operator over elliptic curve}
In this section, we let $k$ be an algebraically closed field of odd characteristic. The aim of this section is to inverstigate the action of the flow operator $\Gr\circ C^{-1}$ on the moduli space $\Sigma_{\Higgs}$ over $k$. 

First fix some notations. A convenient way to choose a $W_2(k)$-lifting of the pair $(C,B)$ is using a bit deformation theory. Recall that we have fixed a $W_(k)$-lifting  $(\P^1_{W_2(k)},\tilde D)$ of $(\P^1_k,D)$. Equip $\P^1$ (resp. $C$) with the log structure defined by the divisor $D$ (resp. $B$). Then the covering map $\pi: (C,B)\to (\P^1,D)$ is in fact a log \'etale morphism over $k$ (where $k$ as well as $W_2(k)$ are equipped with the trivial log structure). Then by deformation theory there exists a unique log \'etale morphism $\tilde \pi: (\tilde C,\tilde B)\to (\P^1_{W_2(k)},\tilde D)$ over $W_2(k)$ lifting $\pi$. So we have a uniquely defined inverse Cartier transform $C^{-1}=C^{-1}_{(C,B)\subset (\tilde C,\tilde B)}$. As before, $\Gr$ refers to the grading with respect to the HN filtration. The flow operator discussed in this section is the usual one (without the presence of a parabolic structure). However, as we shall immediately see, the investigation into the case for $\Sigma_{\Higgs}$ is considerably more difficult than the one for $M_{\Higgs}$.  

Let $(F,\eta)$ represent an element in $\Sigma_{\Higgs}(k)\subset M^s_{\textrm{gr}}(k)$. Then $(V,\nabla)=C^{-1}(F,\eta)$ is $\nabla$-stable, and with trivial determinant. Therefore $\Gr(V,\nabla)$ represents an element in $M_{\textrm{gr}}(k)$. However, if the class of $\Gr(V,\nabla)$ would lie in the the component $M_0$, then it would never flow back to $\Sigma_{\Higgs}$! The aim of this section is to show that such phenomenon cannot happen. Actually, we shall prove that the flow operator induces an endomorphism on $\Sigma_{\Higgs}$. The method is to explore the functoriality of the construction in the flow operator. 

Let $\sM^s_{\theta,\nil}$ (resp. $\sM^{'s}_{\theta,\nil}$) be the moduli functor of rank two stable nilpotent Higgs bundles over $(C,B)/k$ (resp. $(C',B')$). Let $\sM^s_{\nabla,\nil}$ be the moduli functor of rank two stable connections with nilpotent $p$-curvature over $(C,B)/k$. 
\begin{lemma}\label{equivalence of OV}
Notation as above. Then there is an isomorphism of functors $$C_{OV}^{-1}: \sM^{'s}_{\theta,\nil}\to \sM^s_{\nabla,\nil}.$$ 
\end{lemma}
\begin{proof}
Let $M^{'s}_{\theta,\nil}$ be the $k$-scheme corepresenting $\sM^{'s}_{\theta,\nil}$. It has a distinguished $W_2(k)$-lifting $\widetilde{M^{'s}_{\theta,\nil}}$, the moduli scheme corepresenting the moduli functor of rank two stable nilpotent Higgs bundles over $(\tilde C',\tilde B')/W_2(k)$ (\cite[Theorem 1.1]{Lan14}). Let $\coprod_iM_i\to M^{'s}_{\theta,\nil}$ be an \'etale cover such that over each $M_i$, a local universal family $(\sE_i^{\textrm{univ}},\Theta_i^{\textrm{univ}})$ exists (\cite[Theorem 4.7]{Sim94}, \cite[Proposition 7.6, 8.2]{BR85}). By the \'etaleness, there is a distinguished $W_2(k)$-lifting $\tilde M_i$ of $M_i$, induced by the lifting $\widetilde{M^{'s}_{\theta,\nil}}$ of the whole moduli. Then the Frobenius pullback $(C_{M_i},B_{M_i})\times_{M_i,F_{M_i}}M_i$, which is nothing but $(C'_{M_i},B'_{M_i})$, admits a distinguished $W_2(k)$-lifting too, which is the pair
$$
(\tilde C'_{\tilde M_i}:=\tilde C'\times_{W_2(k)}\tilde M_i,\tilde B'_{\tilde M_i}:=\tilde B'\times_{W_2(k))}\tilde M_i).
$$
Set 
$$
(\sC_{\log}/\sM_i:=(C_{M_i},B_{M_i})/M_i,(\tilde C'_{\tilde M_i},\tilde B'_{\tilde M_i})/\tilde M_i).$$ 
Then the relative inverse Cartier transform $C^{-1}_{\sC_{\log}/\sM_i}$ and the relative Cartier transform $C_{\sC_{\log}/\sM_i}$ of Ogus-Vologodsky \cite{OV} and Schepler \cite{Sch08} make an equivalence of categories between the category of $M_i$-flat families of rank two (semi)stable nilpotent Higgs bundles over $(C',B')/k$ and the category of $M_i$-flat families of rank two (semi)stable connections over $(C,B)/k$ with nilpotent $p$-curvature. In particular, $C^{-1}_{\sC_{\log}/\sM_i}$ induces a natural bijection from $\sM^{'s}_{\theta,\nil}(M_i)$ to $\sM^{s}_{\nabla,\nil}(M_i)$. For an arbitrary $T\in (\textrm{Sch}/k)$, we use an descent argument: Let $(\sE,\Theta)\in \sM^{'s}_{\theta,\nil}(T)$. Set $T_i=T\times_{M^{'s}_{\theta,\nil}}M_i$ and $\phi_i: T_i\to M_i$ the induced morphism from the associated moduli map $\phi: T\to M^{'s}_{\theta,\nil}$ to $(\sE,\Theta)$. Then the local flat families $\{\phi_i^*C^{-1}_{\sC_{\log}/\sM_i}(\sE_i^{\textrm{univ}},\Theta_i^{\textrm{univ}})\}_i$ glue to a flat family over $T$. This gives rise to a natural map from to $\sM^{'s}_{\theta,\nil}(T)$ to $\sM^{s}_{\nabla,\nil}(T)$. To show it is a bijection, one repeats the previous argument for the relative Cartier transform, to construct its inverse. It is tedious but straightforward to verify the compatibility of the construction with pullbacks. So we obtain a natural transform $C_{OV}^{-1}: \sM^{'s}_{\theta,\nil}\to \sM^s_{\nabla,\nil}$, which is an isomorphism of moduli functors.    
\end{proof}
By abuse of notation, we denote again by $C_{OV}^{-1}: M^{'s}_{\theta,\nil}\to M^{s}_{\nabla,\nil}$ the induced $k$-isomorphism from the above lemma. For any $T\in (\textrm{Sch}/k)$, we consider the following Cartesian diagram:
$$
 \xymatrix{C'_T\ar[r]^-{F_T}\ar[d]_-{}&C_T\ar[d]^-{}\\
		T\ar[r]^-{F_T}&T,}
  $$
where $F_T: T\to T$ is the absolute Frobenius. For $[(\sE,\Theta)]\in \sM^s_{\theta,\nil}(T)$, $F_T^*(\sE,\Theta)$ represents an element in $\sM^{'s}_{\theta,\nil}(T)$. Clearly, it defines a natural transformation $F: \sM^s_{\theta,\nil}\to \sM^{'s}_{\theta,\nil}$. 
\begin{proposition}
Fix a $W_2(k)$-lifting $(\tilde C,\tilde B)$ of $(C,B)$ as before. Then the composite of natural transformations 
$$
C_{OV}^{-1}\circ F: \sM^s_{\theta,\nil}\to \sM^{s}_{\nabla,\nil}
$$
induces a purely inseparable $k$-morphism 
$$
C^{-1}: M^s_{\theta,\nil}\to M^{s}_{\nabla,\nil}.
$$ 
\end{proposition}
\begin{proof}
Let $M^{s}_{\theta,\nil}$ be the moduli space corepresenting $\sM^{s}_{\theta,\nil}$. Note that $M^{'s}_{\theta,\nil}\cong M^{s}_{\theta,\nil}\times_{k,F_k}k$. Because the absolute Frobenius commutes with arbitrary pullback, one verifies easily that the representing morphism to $F$ is naturally identified with the relative Frobenius $F_{M^{s}_{\theta,\nil}/k}: M^{s}_{\theta,\nil}\to M^{'s}_{\theta,\nil}$. By Lemma \ref{equivalence of OV}, $C^{-1}_{OV}$ is an $k$-isomorphism. Hence the composite morphism $C^{-1}$ is purely inseparable. 
\end{proof}

Let $(R,\mathfrak{m})$ be a valuation ring containing $k$ which is isomorphic to its residue field. Let $\xi_1$ (resp. $\xi_0$) be the generic (resp. special) point of $T=\Spec(R)$. Let $\sV$ be a $T$-flat family of rank two vector bundles with trivial determinant over $C$.  
	\begin{lemma}\label{HN is constant}
		Notation as above. The Harder-Narasimhan polygon of $\sV_{\xi_0}$ equals that of $\sV_{\xi_1}$.    
	\end{lemma}
	\begin{proof}
	Let $K$ be the fractional field of $R$.	Let $\bar \xi_1$ be the geometric generic point of $T$. If $\sV_{\bar \xi_1}$ is semistable, then we may write 
		\begin{equation*}
			0\to \sO_{C_{\bar K}}\otimes L\to \sV_{\bar \xi_1} \to \sO_{C_{\bar K}}\otimes L\to 0,
		\end{equation*}
  where $L$ is torsion of order two. By the properness of Quot-scheme, $\sV_{\xi_0}$ has a locally free quotient of degree zero, and hence is semistable. If $\sV_{\bar \xi_1}$ is unstable, then we may write
		\begin{equation*}
			0\to L_{\bar \xi_1}\to \sV_{\bar \xi_1} \to L_{\bar \xi_1}^{-1}\to 0,
		\end{equation*}
		with $\deg(L_{\bar \xi_1})>0$. Again by the properness of Quot-scheme, it follows that $\sV_{\xi_0}$ contains a sub line bundle $L_{\xi_0}$ of degree equal to $\deg(L_{\bar \xi_1})$. Hence $\sV_{\xi_0}$ is unstable. 
		Let $M_{\xi_0}$ be the maximal destabilizer of $\sV_{\xi_0}$. For the degree reason, the composite 
		\begin{equation*}
			L_{\xi_0}\to \sV_{\xi_0}\to \sV_{\xi_0}/M_{\xi_0}
		\end{equation*}
		has to be zero. Therefore $L_{\xi_0}\subset M_{\xi_0}$. Since both are sub line bundles, it follows that $L_{\xi_0}=M_{\xi_0}$. Summarizing these two cases, one concludes the lemma.
	\end{proof}

Let $\sM^s_{\nabla,\det=\sO}$ be the moduli functor of rank two stable connections over $(C,B)/k$ with trivial determinant. Let $\sM_{\textrm{gr}}$ (resp. $\sM^s_{\textrm{gr}}$) be the moduli functor of rank two semistable (resp. stable) graded Higgs bundles over $(C,B)/k$ with trivial determinant. We construct a natural transformation 
$$
\Gr: \sM^s_{\nabla,\det=\sO}\to \sM_{\textrm{gr}}
$$
as follows: Let $M^s_{\nabla,\det=\sO}$ be the $k$-scheme corepresenting $\sM^s_{\nabla,\det=\sO}$. It is nonsingular. Let $[(V,\nabla)]\in M^s_{\nabla,\det=\sO}(k)$. Ignoring the connection, the rank two bundle $V$ possibly loses the stability. As local universal families for the functor $\sM^s_{\nabla,\det=\sO}$ exist, the schematic Harder-Narasimhan stratification (ignoring connections) on $M^s_{\nabla,\det=\sO}$ exists (\cite{N11}). By Lemma \ref{HN is constant} and \cite[Corollary 6]{N11}, each connected component $M^s_{\nabla,\det=\sO}$ has only one HN strata. Therefore, for any $T\in (\textrm{Sch}/k)$ and any $(\sV,\nabla)$ representing an element in $\sM^s_{\nabla,\det=\sO}(T)$, the relative HN filtration on $\sV$ exists. Taking the grading of $(\sV,\nabla)$ with respect to the relative HN filtration, we obtain $\Gr(\sV,\nabla)$ which represents an element in $\sM_{\textrm{gr}}(T)$. The compatibility of this construction with pullback follows from the uniquenss of the relative HN filtration. By abuse of notation, we denote again by $\Gr: M^s_{\nabla,\det=\sO}\to M_{\textrm{gr}}$ the induced $k$-morphism on moduli spaces from the construction. Assembling the above results, we obtain the following result.

\begin{proposition}\label{morphism for M}
Notation as above. Let $\phi: M^s_{\textrm{gr}}\to M_{\textrm{gr}}$ be the $k$-morphism representing the composite of the following natural transformations
$$
\sM^s_{\textrm{gr}}\to \sM^s_{\theta,\nil,\det=\sO}\stackrel{C^{-1}}{\longrightarrow} \sM^{s}_{\nabla,\nil,\det=\sO}\to \sM^{s}_{\nabla,\det=\sO}\stackrel{\Gr}{\longrightarrow} \sM_{\textrm{gr}},
$$
where $\sM^s_{\theta,\nil,\det=\sO}$ (resp. $\sM^{s}_{\nabla,\nil,\det=\sO}$) is the submoduli functor of $\sM^s_{\theta,\nil}$ (resp. $\sM^{s}_{\nabla,\nil}$) containing objects with trivial determinant. Then $\phi$ is inseparable, and for $[(F,\eta)]\in M^s_{\textrm{gr}}(k)$, it holds that
$$
\phi([F,\eta])=[\Gr\circ C^{-1}(F,\eta)].
$$
\end{proposition}
There is a reasonable way to extend $\phi$ to $M_{\textrm{gr}}$: Suppose a $T$-flat family belonging to $\sM_{\textrm{gr}}(T)$ with $T$ connected has one geometric fiber lying in $M^s_{\textrm{gr}}$. Then the associated moduli map $T\to M_{\textrm{gr}}$ factors through $M^s_{\textrm{gr}}\hookrightarrow M_{\textrm{gr}}$. We shall regard the disjoint component $M_0=M_{\textrm{gr}}-M^s_{\textrm{gr}}$, containing the $S$-equivalence classes of strictly semistable objects, as the parameter space for lower rank objects. 
\begin{definition}\label{flow morphism on the whole moduli}
Notation as above. Define the flow morphism $\phi: M_{\textrm{gr}}\to M_{\textrm{gr}}$ by requiring $\phi|_{M^s_{\textrm{gr}}}$ be the morphism constructed in Proposition \ref{morphism for M}, and $\phi|_{M_0}=\Pic^0(C)\to \Pic^0(C)=M_0$ be the multiplication by $p$.  
\end{definition}
It is vivid to call $M_0$ the \emph{black hole} of the moduli $M_{\textrm{gr}}$: Once a point flows into $M_0$, it will never leave $M_0$. This is clear by the definition of $\phi$, which is a manifestation of the black-hole principle: Strictly semistable objects flow only to strictly semistable objects. One may wonder if $\phi$ maps $M^s_{\textrm{gr}}$ into $M^s_{\textrm{gr}}$. However, this is not the case: In fact, one can show that $\phi$ maps $M_2$ either to itself or to $M_0$ (these logarithmic Higgs bundles are essentially the uniformizing Higgs bundle of $\P^1$ removing four points. See \cite[Proposition 3.1]{LS}). 
\begin{proposition}\label{flow preserves M_1}		
The flow morphism $\phi$ preserves $M_1$. Consequently, an element in $\Sigma_{Higgs}(k)$ will never flow into the strictly semistable locus.   		
\end{proposition}
\begin{proof}
Since $M_1$ is irreducible, it follows that $\phi(M_1)$ must be irreducible. Recall that in Lemma \ref{four fibers}, we have constructed four distinguished fibers of the ruled surface $\pi_1: M_1\to C$. Pick any $Q\in |B|$. It suffices to show that $\phi$ preserves the fiber $\pi_1^{-1}(Q)$. 

Set $P=\pi(Q)$. Let $[(F,\eta)]\in \pi_1^{-1}(Q)$. Then $(E,\theta)=\pi_{\Par*}(F,\eta)$ represents an element in $M_{\frac{1}{2}P}$ after Lemma \ref{four fibers}. Because of our choice of $W_2(k)$-lifting for $(C,B)$, Lemma \ref{parabolic pullback commutes with flow operator} implies that
$$
\Gr\circ C^{-1}\pi_{\Par}^*(E,\theta)\cong \pi_{\Par}^*\Gr\circ C^{-1}(E,\theta).
$$
By the discussion at the end of \S3, we have $[\Gr\circ C^{-1}(E,\theta)]\in M_{\frac{1}{2}P}$. Then the last isomorphism shows that $\phi([(F,\eta)])\in \pi_1^{-1}(Q)$, as desired.
\end{proof}
	
We can even find some fixed points on the four distinguished fibers. For that, we label four two-torsion points by $\{Q_i\}_{1\leq i\leq 4}$ and set $Q_1$ to be the origin. Note that $(\sO_{C}(Q_i)\oplus \sO_{C}(-Q_i),\theta_{2[Q_i+Q_j]}))$ represents an intersection point of $\pi_1^{-1}(Q_i)$ with $\Sigma_{\mathrm{Higgs}}^{Q_j}$ (see Lemma \ref{four sections}). 
\begin{proposition}\label{intersection points are fixed}
		Notation as Proposition \ref{flow preserves M_1}. Then $\phi$ fixes the intersection points $\pi_1^{-1}(Q_i)\cap \Sigma_{\mathrm{Higgs}}^{Q_j}$ for each $i,j$.
	\end{proposition}
	\begin{proof}
		Recall that the Legendre family, with affine equation $y^2=x(x-1)(x-t)$, has two semistable fibers at $t=0,1$, one quasi-semistable fiber at $t=\infty$ and is smooth elsewhere. Its associated logarithmic Kodaira-Spencer bundle is of form
		\begin{equation*}
			\sO_{\Pl}\stackrel{\theta_{\lambda}}{\longrightarrow}\sO_{\Pl}(-1)\otimes \Omega_{\P^1}(D).
		\end{equation*}
		Here we add artificially a zero on the Higgs field at $\lambda$, because we add one more pole $\lambda$ in the logarithmic differential form. Let $P_i$ be the image of $Q_i$ under the projection $\pi$. Using a simple automorphism of $\P^1$ as base change, we pullback the Legendre family to a family which has the unique quasi-semistable (but not semistable) fiber at $Q_i$ and whose associated Kodaira-Spencer map has one zero at $[Q_i+Q_j]$ with $j\neq 1$. 
		
		Consider the associated canonical parabolic Kodaira-Spencer system to the new family. It represent an element in $M_{\frac{1}{2}P_i}$. By the Higgs periodicity theorem (Theorem \ref{Higgs periodicity theorem}), it is periodic. 
		Since the denominator of a nontrivial parabolic weight is two, it is actually one-periodic. Therefore, its parabolic pullback, which represents an intersection point in the proposition, is also one-periodic. In other words, its isomorphism class is fixed by $\phi$. 
		For the remaining intersection points contained in $\Sigma_{\mathrm{Higgs}}$, one may artificially put the parabolic structure at the same point as the zero of the Kodaira-Spencer map. 
		By Proposition \ref{identification of self-maps}, they are also one-periodic.           
	\end{proof}
	
Pick any $P\in C(k)$ and form the graded logarithmic Higgs bundle $(F,\eta)$
$$
\eta_{2\sigma(P)}: \sO_C(P)\to \sO_{C}(-P)\otimes \Omega_C(B).
$$
Let $(F',\eta')=\Gr\circ C^{-1}(F,\eta)$. By Proposition \ref{flow preserves M_1}, there exists some $P'\in C(k)$ and an effective Weil divisor $R'+S'$ of length two satisfying $[R'+S']=[2\sigma(P')]$, such that $(F',\eta')$ is isomorphic to 
$$
\eta_{R'+S'}: \sO_C(P')\to \sO_{C}(-P')\otimes \Omega_C(B).
$$
We are going to show $R'=S'$. The geometry of the flow morphism acting on the ruled surface $M_1$ plays an important role. 

Set $(V,\nabla)=C^{-1}(F,\eta)$. Then $V\cong F'$. Clearly, 
$$
(V,\nabla)[G]:=(V,\nabla)\oplus \sigma^*(V,\nabla)\cong C^{-1}((F,\eta)[G]).
$$
Since $C^{-1}$ is an exact functor, it follows from Proposition \ref{ses of E[G]} that we have a short exact sequence of logarithmic connections:
\begin{equation}\label{s.e.s. of V[G]}
	0 \to (V_{1}, \nabla_{1}) \to (V, \nabla)[G] \to (V_{2}, \nabla_{2}) \to 0.
	\end{equation}
Here 
\begin{equation*}
	\begin{array}{rcl}
		(V_1,\nabla_1)&\cong &C^{-1}\pi_{\Par}^*(\sO\oplus \sO(-1),\theta_{x})\\
		&\cong& \pi_{\Par}^*C^{-1}(\sO\oplus \sO(-1),\theta_{x}) \\
		&\cong& \pi^*C^{-1}(\sO\oplus \sO(-1),\theta_{x}).
	\end{array}
	\end{equation*}
In above, $x=\pi(P)$, and the last isomorphism holds because the parabolic structure on $\sS$ in Proposition \ref{ses of E[G]} is trivial. Then
$$
(V_2,\nabla_2)\cong \pi^*C^{-1}(\sO\oplus \sO(-1),\theta_{x})\otimes (\pi^*\sO(p),\nabla_{can}).
$$
The bundles $V[G],V_1,V_2$ are all unstable (and hence decomposable). Let $HN_{i}$ (resp. $HN$) be the maximal destabilizer of $V_{i}$ (resp. $V[G]$).  
\begin{lemma}\label{compatibility of HN}
	Notation as above. Then 
$$
HN\cap V_1=HN_1;\quad HN/HN_1=HN_2.
$$
\end{lemma}
\begin{proof}
Consider first the intersection $V_1\cap HN\subset V[G]$. It is saturated because both $HN_1$ and $V_1$ are saturated. Note that the degrees of $ HN $ and $ V_{1} $ are $ 2 $ and $ -2p $ respectively. So $\rank (V_1\cap HN)<2$. Since $\deg V{G} =0\neq \deg V_1+\deg HN$, it follows $\rank (V_1\cap HN) \neq 0$. Consequently, $ V_1 \cap HN $ is a sub line bundle and we denote it by $ L_{1} $. Fix a splitting $V[G] \cong HN \oplus HN^{\perp}$ of the short exact sequence
$$
0\to HN\to V[G]\to V[G]/HN\to 0.
$$
Running the same argument for $V_1\cap HN^{\perp}$, one gets a second sub line bundle $L_{2}$ of $V_1$. Clearly $V_1 \cong L_1\oplus L_2 $. Juxtaposing $L_1 \oplus L_2 \cong HN_{1}\oplus HN_{1}^{\perp}$, we find that $\deg L_1\neq \deg L_2$, and moreover
$$
\max \left\lbrace \deg L_1, \deg L_2 \right\rbrace =\deg HN_1=-p+1, \min \left\lbrace \deg L_1, \deg L_2 \right\rbrace =\deg HN_1^{\perp}=-p-1.
$$ 
We claim $L_{1} = HN_{1}$. Consider the set $\mathbb S$ of all subbundles $H$ in $V{G}$ which split the above exact sequence. It is not difficult to see that the subbundle $\cap_{H\in \mathbb S}H$ is simply zero. It follows $\cap_{H\in \mathbb S}(H \cap V_{1})=0$. In other words, all sub line bundles inside $HN$ which split the following exact sequence
$$
0\to L_1\to HN\to HN/L_1\to 0.
$$
intersect to zero. It follows that $L_2$ cannot be the maximal destabilizer. Thus the claim is proved. By taking the quotient, we obtain a sub line bundle $HN/HN_1\subset V[G]/HN_1=V_2$. It is of degree $(p+1)$, hence equals the maximal destabilizer of $V_{2}$. This completes the proof. 
\end{proof}
\begin{proposition}\label{double zero}
Let $(F,\eta)$ and $(F',\eta')$ be as above. Then
$$
[(F',\eta')]\in \bigcup_{i,j}\pi_1^{-1}(Q_i)\cup \Sigma^{Q_j}_{\mathrm{Higgs}}.
$$
\end{proposition}
\begin{proof}
Thanks to Lemma \ref{parabolic pullback commutes with flow operator} and Lemma \ref{compatibility of HN}, taking the grading functor to \ref{s.e.s. of V[G]} yields the following short exact sequence:
\begin{equation*}\label{the short exact sequence after flow operator}
			0 \to \pi^{*}((E, \theta_{x'})\otimes (\sO(-m),0)) \to (F',\eta')[G] 
			\to \pi^{*}((E, \theta_{x'})\otimes (\sO(m+1),0)) \to 0,
\end{equation*}
where $E=\sO\oplus \sO(-1)$, $x'=\varphi_{\lambda,p}(x)$ and $m=\frac{p-1}{2}$.

Consider the $(1,0)$-part of the Higgs structure in the above exact sequence, which may be written as follows:
		\begin{equation*}
			\begin{tikzcd}
				0 \arrow[r] & HN_{1} \arrow[r]\arrow[d,"\pi^{*}\theta_{x'}"]& HN \arrow[r]\arrow[d,"\eta' \oplus \sigma^*{\eta'}"] &HN_{2} \arrow[r]\arrow[d,"\pi^{*}\theta_{x'}"]& 0 \\
				0 \arrow[r]& HN_{1}^{\perp} \otimes \Omega_{C}(B) \arrow[r]& HN^{\perp} \otimes\Omega_{C}(B) \arrow[r]&HN_{2}^{\perp}\otimes\Omega_{C}(B)\arrow[r] & 0.
			\end{tikzcd}
		\end{equation*}
		So we obtain the following commutative diagram:
		\begin{equation*}
			\begin{tikzcd}
				HN_{1} \otimes HN_{2} \arrow[r,"\cong"]\arrow[d,"(\pi^{*}\theta_{x'})^{\otimes 2}"'] & \det HN \arrow[d,"\det(\eta' \oplus \sigma^*{\eta'})"] \\
				HN_{1}^{\perp} \otimes HN_{2}^{\perp} \otimes (\Omega_{C}( B))^{\otimes 2} \arrow[r,"\cong"] & \det( HN^{\perp}  \otimes (\Omega_{C}(B))).
			\end{tikzcd}
		\end{equation*}
		We deduce then the following equality of Weil divisors:
$$
2\pi^*x'=\mathrm{div}((\pi^{*}\theta_{x'})^{\otimes 2})=\mathrm{div}(\det(\eta' \oplus \sigma^*{\eta'}))=R'+S'+\sigma(R')+\sigma(S').
$$
Comparing the images of their supports under $\pi$, we immediately conclude that either  $R'=\sigma(S')$ or $R'=S'$ holds (or both hold). In the former case, we have
$$
[2\sigma(P')]=[R'+S']=[\sigma(S')+S'],
$$
which is the origin. So $P'$ is a two torsion point. Then $[(F',\eta')]\in \cup_i\pi_1^{-1}(Q_i)$ in this case. In the latter case, $\eta'^{1,0}$ vanishes only at one point. So by Lemma \ref{four sections}, $[(F',\eta')]\in \cup_j\Sigma_{\Higgs}^{Q_j}$. The proof is completed.  
\end{proof}
Finally, we can prove the main result of the section. The following picture gives a good orientation of the situation. 

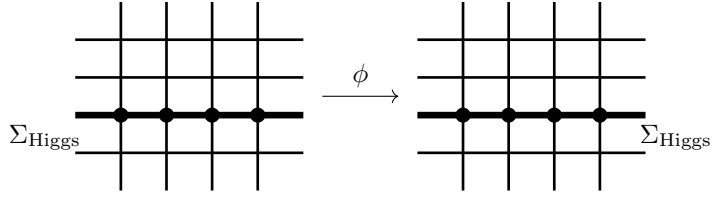
\begin{center}
		\begin{tikzpicture}
		
		\draw [line width=2.5pt](-1.2,0) -- (1.8,0);
		\node[below] at (-1.6,0) {$ \Sigma_{\Higgs} $};
		\draw [line width=1pt](-1.2,0.5) -- (1.8,0.5);
		\draw [line width=1pt](-1.2,-0.5) -- (1.8,-0.5);
		\draw [line width=1pt](-1.2,1) -- (1.8,1);
		\draw [line width=1pt](0,1.5) -- (0,-1);
		\draw [line width=1pt](0.6,1.5) -- (0.6,-1);
		\draw [line width=1pt](-0.6,1.5) -- (-0.6,-1);
		\draw [line width=1pt](1.2,1.5) -- (1.2,-1);
		\fill (0,0) circle (.1);
		\fill (-0.6,0) circle (.1); 
		\fill (0.6,0) circle (.1); 
		\fill (1.2,0) circle (.1);

		\draw [line width=2.5pt](3.3,0) -- (6.3,0);
		\node[below] at (6.7,0) {$ \Sigma_{\Higgs} $};
		\draw [line width=1pt](3.3,0.5) -- (6.3,0.5);
		\draw [line width=1pt](3.3,-0.5) -- (6.3,-0.5);
		\draw [line width=1pt](3.3,1) -- (6.3,1);
		\draw [line width=1pt](4.5,1.5) -- (4.5,-1);
		\draw [line width=1pt](5.1,1.5) -- (5.1,-1);
		\draw [line width=1pt](3.9,1.5) -- (3.9,-1);
		\draw [line width=1pt](5.7,1.5) -- (5.7,-1);
		\fill (4.5,0) circle (.1);
		\fill (5.1,0) circle (.1); 
		\fill (3.9,0) circle (.1); 
		\fill (5.7,0) circle (.1); 
		
		\node [above] at (2.55,0.25) {$\phi$};
		\draw [->] (2.05,0.25) -- (3.05,0.25);
		\end{tikzpicture}
	\end{center}

\begin{corollary}\label{flow morphsim preserving Sigma}
The flow morphism $\phi$ preserves $\Sigma_{\mathrm{Higgs}}$.     
\end{corollary}
\begin{proof}
For any closed point $P\in C$, consider the composite 
$$
\P^1\cong \pi_1^{-1}(P)\to M_1\stackrel{\phi}{\to} M_1\stackrel{\pi_1}{\to} C.
$$
It has to be constant. Hence by rigidity lemma (\cite[Section 4, page 43]{Mum70}), $\phi$ induces a morphism $\bar \phi: C\to C$ making the following diagram commute:
$$
 \xymatrix{M_1\ar[r]^-{\phi}\ar[d]_-{\pi_1}&M_1\ar[d]^-{\pi_1}\\
		C\ar[r]^-{\bar \phi}&C.}
  $$
By Proposition \ref{intersection points are fixed}, $\bar \phi$ is non-constant. Therefore, the composite  
$$
\Sigma_{\Higgs}\hookrightarrow M_1\stackrel{\pi_1}{\to} C\stackrel{\bar \phi}{\to} C
$$
is non-constant. It follows that $\phi(\Sigma_{\Higgs})$, which is an irreducible curve in $M_1$, cannot be a closed fiber. So $\phi^{-1}(\phi(\Sigma_{\Higgs})\cap \cup_i\pi_1^{-1}(Q_i))$ is a proper closed subset of $\Sigma_{\Higgs}$. Let $\Sigma_{\Higgs}^0$ be its complement, which is open and hence irreducible. So there is some $Q_{j_0}$ such that $\phi(\Sigma_{\Higgs}^0)\subset \Sigma_{\Higgs}^{Q_{j_0}}$. Since $\phi$ is continuous in Zariski topology, it follows that $\phi(\Sigma_{\Higgs})\subset \Sigma_{\Higgs}^{Q_{j_0}}$. Using Proposition \ref{intersection points are fixed} again, we know that $Q_{j_0}$ must be the origin $Q_1$. This completes the whole proof.
\end{proof}

 \section{Proof of Sun-Yang-Zuo conjecture}
 In this section, we prove Theorem \ref{main result}. The major input, obtained from the discussions in \S3-5, is a moduli interpretation of the diagram \ref{conjecture}. 
\begin{proposition}\label{decent}
Notations as in \S3-\S5. Write the restriction of the flow morphism $\phi$ to $\Sigma_{\Higgs}$ by $\phi_{\lambda,p}$. Then the following diagram commutes	
		\begin{equation}\label{commutitivity of flow operator}
			\begin{tikzcd}
				\Sigma_{\mathrm{Higgs}} \arrow[r,"\phi_{\lambda,p}"] 
				\arrow[d,"\pi_{\mathrm{Higgs}}"'] & \Sigma_{\mathrm{Higgs}} 
				\arrow[d,"\pi_{\mathrm{Higgs}}"] \\
				M_{\mathrm{Higgs}} \arrow[r,"\varphi_{\lambda,p}"] & M_{\mathrm{Higgs}},
			\end{tikzcd}  
		\end{equation}	
  where $\varphi_{\lambda,p}$ is the flow morphism of Sun-Yang-Zuo.
	\end{proposition}
	 \begin{proof}
For $[(F(P)=\sO_C(P)\oplus \sO_C(-P),\eta_{2\sigma(P)})] \in \Sigma_{\mathrm{Higgs}}(k)$, the proof of Corollary \ref{max destabilizer} tells us that $\pi_{\Higgs}([(F(P),\eta_{2\sigma(P)})] )$ is represented by
$(E,\theta_{x})$ with $x=\pi(P)$. Put $\varphi_{\lambda,p}([E,\theta_x])=[(E,\theta_{x'})]$. Write 
$$
(F(P')=\sO_C(P')\oplus \sO_C(-P'),\eta'):=\Gr\circ C^{-1}(F(P),\eta_{2\sigma(P)}).
$$
By Corollary \ref{flow morphsim preserving Sigma}, $\eta'=\eta_{2\sigma(P')}$. Now reading off the proof of Proposition \ref{double zero}, we actually obtain
$$
2\pi^*x'=2R'=2\sigma(P').
$$
Therefore $\pi^*x'=\sigma(P')$, and hence $x'=\pi(P')$. Finally, we get
$$
\pi_{\Higgs}([(F(P'),\eta_{2\sigma(P')})])=[(E,\theta_{x'})],
$$
as desired.
\end{proof}	
 
\begin{lemma}\label{arithmetic lemma}
Let $k$ be an algebraically closed field of odd characteristic $p$, and $C$ an elliptic curve over $k$. Let $\psi:C \rightarrow C $ be an isogeny satisfying:
		\begin{itemize}
			\item [(i)] $ \psi $ is inseparable of degree $ p^{2} $;
			\item [(ii)] $ \psi $ fixes $ 2 $-torsion points;
			\item [(iii)] $ \psi $ sends $ p $-torsion points to the origin.
		\end{itemize}
		Then $ \psi$ is the multiplication map by $p$ up to sign.
	\end{lemma}
	
 \begin{proof}
Let $[p]:C\to C$ be the multiplication by $p$ map. It is purely inseparable when $C$ is supersingular. In this case, we let $\sD$ be the quaternion algebra over $\Q$ ramified at $ p,\infty $. Then $\mathrm{End}_k(C) = \OO $ is a maximal order of $\mathcal{D}$. Let $u=\psi/p \in \mathcal{D} $. Claim that $u \in \OO^{\times}$. It suffices to show $u \in \OO $. This is because 
$$ 
\overline{u}=\mathrm{tr}(u)-u \in \OO, \quad u\cdot \overline{u}=\frac{\psi\cdot \overline{\psi}}{p^{2}}=1. 
$$ 
Let $ \OO_{l}=\OO \otimes_{\mathbb{Z}} \mathbb{Z}_{l} $, where $ l $ is a prime number. It is sufficient to show $ u \in \OO_{l} $ for all $ l $. For prime $ l \neq p $, then $ p \in \OO_{l}^{\times} $, hence $ u=\psi/p $ lies in $ \OO_{l} $ trivially. For $ l=p $, since $ \mathcal{D} $ is ramified at $p$, it follows that $\OO_{p}$ is the unique maximal order in $ \mathcal{D}_{p} $, which is characterized by 
$$ 
\OO_{p}=\left\lbrace x \in \mathcal{D}_{p}^{\times}| x \cdot \overline{x} \in \mathbb{Z}_{p}\right\rbrace \cup \left\lbrace 0 \right\rbrace.
$$ 
But since $ u\cdot \overline{u}=1$, we infer that $ u \in \OO_{p} $. This proves our claim. So we obtain that $\psi=u\circ [p]$ for some automorphism $u$. When $C$ is ordinary, $[p]=F\circ V=V\circ F$, where $F$ (the relative Frobenius) is inseparable of degree $p$ and $V$ (the Verschiebung) is separable of degree $p$. By condition (i), $\psi$ factors as  $\psi=F\circ \bar \psi$ for isogeny $\bar \psi$. Because of condition (iii), $\psi$ cannot be inseparable. By the same condition, there exists an automorphism $u$ of $C$ such that $\bar \psi=u\circ \bar \psi$ (see \cite[Theorem 4, \S7]{Mum70}). Composing with $F$ on both sides, we also obtain $\psi=u\circ [p]$ for some automorphism $u$. Since $[p]$ fixes all two torsion points, condition (ii) on $\psi$ implies that $u$ fixes all two torsion points too. Now as $p\geq 3$, one checks immediately that the only automorphism fixing all two torsion points is $\pm 1$ ($-1$ refers to the canonical involution). The lemma is proved.
\end{proof}
We proceed to the proof of Theorem \ref{main result}
\begin{proof}
Clearly, the theorem follows if one shows $\phi_{\lambda,p}$ satisfies all three conditions of $\psi$ in Lemma \ref{arithmetic lemma}. By Proposition \ref{morphism for M}, it is inseparable. This can also be seen by the commutativity of the diagram \ref{commutitivity of flow operator} and the fact that $\varphi_{\lambda,p}$ is inseparable as proved by Sun-Yang-Zuo. In fact, by \cite[\S4.3]{SYZ}, one can even infer that $\deg(\phi_{\lambda,p})=p^2$. This is condition (i). Condition (ii) is proved in Proposition \ref{intersection points are fixed}. For condition (iii),  we only have to check $ \phi_{\lambda,p}$ is not purely inseparable in the ordinary case. But then, it is equivalent to show $\varphi_{\lambda,p}$ of Sun-Yang-Zuo is not purely inseparable in this case. We shall prove it in the lemma below. Therefore, we have completed the proof of Theorem \ref{main result}. 
\end{proof}
	 
In \cite[page 42]{SYZ}, $\varphi_{\lambda,p}$ has an explicit formula in affine coordinate:
\begin{equation*}
		\varphi_{\lambda,p}(x) = \frac{x^{p}}{\lambda^{p-1}}\cdot (\frac{f_{\lambda}(x^{p})}{g_{\lambda}(x^{p})})^2 ,
\end{equation*}
where $f_{\lambda}$ is a monic polynomial of degree $\frac{p-1}{2}$, $g_{\lambda}$ is of degree $\frac{p-1}{2}$ too but with leading coefficient $\det A(\lambda)$ given by
\begin{equation*}
			A(\lambda)=\begin{bmatrix}
				\frac{\lambda-{\lambda}^{p}}{1} & \frac{{\lambda}^{2}-{\lambda}^{p}}{2} & \cdots & \frac{{\lambda}^{\frac{p-1}{2}}-{\lambda}^{p}}{\frac{p-1}{2}} \\
				\frac{{\lambda}^{2}-{\lambda}^{p}}{2} & \frac{{\lambda}^{3}-{\lambda}^{p}}{3} & \ddots & \vdots\\
				\vdots & \ddots & & \\
				\frac{{\lambda}^{\frac{p-1}{2}}-{\lambda}^{p}}{\frac{p-1}{2}} & \cdots & \frac{{\lambda}^{p-3}-{\lambda}^{p}}{p-3} & \frac{{\lambda}^{p-2}-{\lambda}^{p}}{p-2}
			\end{bmatrix}_{\frac{p-1}{2} \times \frac{p-1}{2}}
		\end{equation*}
\begin{lemma}
Suppose $C_{\lambda}$ is ordinary. Then $\varphi_{\lambda,p}$ is not purely inseparable. 
\end{lemma}
\begin{proof}
Actually, we shall establish the following formula:
\begin{equation}\label{formula for det}
			\det A(\lambda)=c \cdot \lambda^{(\frac{p-1}{2})^{2}}  \cdot (1-\lambda)^{(\frac{p-1}{2})^{2}} \cdot H_{p}(\lambda),
		\end{equation}
where $H_{p}(\lambda)$ is the Hasse polynomial, and $ c $ is a non-zero constant. Since $C_{\lambda}$ is ordinary, $H_{p}(\lambda)\neq 0$. The lemma follows. 
 
In the supersingular case, $\phi_{\lambda,p}$ is purely inseparable as $[p]$ is purely inseparable. It implies that $H_{p}(\lambda)$ divides $\det A(\lambda)$. We continue to calculate the vanishing order of $\det A(\lambda)$ at $\lambda=0 $ and $\lambda=1$. 
		
Looking at a general term in the Laplace expansion of the determinant, one observes quickly that
$$
\deg(\det A(\lambda))\leq \frac{p(p-1)}{2},\quad \lambda^{\frac{(p-1)^2}{4}}|\det A(\lambda).
$$
Make a substitution $\lambda=t+1$. Notice that ($1\leq n\leq \frac{p-1}{2}$)
\begin{equation*}
			\frac{\lambda^{n}-\lambda^{p}}{n}=\frac{1-(1+t)^{n}}{n}+\frac{t^{p}}{n}.
\end{equation*}
In order to determine the vanishing order of $\det A(t+1)$ at $t=0$, we may ignore the higher order term $\frac{t^{p}}{n}$ in each entry. Claim: the determinant of the following matrix
		\begin{equation*}
			B(t)=\begin{bmatrix}
				\frac{(t+1)-1}{1} & \frac{(t+1)^{2}-1}{2} & \cdots & \frac{(t+1)^{\frac{p-1}{2}}-1}{\frac{p-1}{2}} \\
				\frac{(t+1)^{2}-1}{2} & \frac{(t+1)^{3}-1}{3} & \ddots & \vdots\\
				\vdots & \ddots & & \\
				\frac{(t+1)^{\frac{p-1}{2}}-1}{\frac{p-1}{2}} & \cdots & \frac{(t+1)^{p-3}-1}{p-3} & \frac{(t+1)^{p-2}-1}{p-2}
			\end{bmatrix}_{\frac{p-1}{2} \times \frac{p-1}{2}}
		\end{equation*}
		equals $ c \cdot t^{\frac{(p-1)^{2}}{4}} $ for some non-zero constant $c$. Here is a trick: The derivative $B'(t)$ equals 
	\begin{equation*}
			\begin{split}
				B'(t) &= \begin{bmatrix}
					1 & (t+1)^{1} & \cdots & (t+1)^{\frac{p-3}{2}} \\
					(t+1)^{1} & (t+1)^{2} & \ddots & \vdots\\
					\vdots & \ddots & & \\
					(t+1)^{\frac{p-3}{2}} & \cdots & (t+1)^{p-4} & (t+1)^{p-3}
				\end{bmatrix}_{\frac{p-1}{2} \times \frac{p-1}{2}} \\ & = \begin{bmatrix} 1 \\ (t+1)^{1} \\ \vdots \\(t+1)^{\frac{p-3}{2}} \end{bmatrix}\cdot \begin{bmatrix}
					1 & (t+1)^{1} & \cdots & (t+1)^{\frac{p-3}{2}}
				\end{bmatrix}
			\end{split}
		\end{equation*}
		Let 
		\begin{equation*}
			[1 \; (t+1) \; \cdots \; (t+1)^{\frac{p-3}{2}}]=[1 \; t \; \cdots \; t^{\frac{p-3}{2}}] \cdot P
		\end{equation*}
 for some $P\in \Gl(k)$. Then 
		\begin{equation*}
			B'(t)=P^{t} \cdot \begin{bmatrix}
				1 & t^{1} & \cdots & t^{\frac{p-3}{2}} \\
				t^{1} & t^{2} & \ddots & \vdots\\
				\vdots & \ddots & & \\
				t^{\frac{p-3}{2}} & \cdots & t^{p-4} & t^{p-3}
			\end{bmatrix}_{\frac{p-1}{2} \times \frac{p-1}{2}}	
			\cdot P
		\end{equation*}
Taking the formal integration of both sides, we obtain the equality: 
		\begin{equation*}
			B(t)=P^{t} \cdot \begin{bmatrix}
				t & \frac{t^{2}}{2} & \cdots & \frac{t^{\frac{p-1}{2}}}{\frac{p-1}{2}} \\
				\frac{t^{2}}{2} & \frac{t^{3}}{3} & \ddots & \vdots\\
				\vdots & \ddots & & \\
				\frac{t^{\frac{p-1}{2}}}{\frac{p-1}{2}} & \cdots & \frac{t^{p-3}}{p-3} & \frac{t^{p-2}}{p-2}
			\end{bmatrix}_{\frac{p-1}{2} \times \frac{p-1}{2}}	
			\cdot P
		\end{equation*}
Then it follows that $ \det B(t) =(\det P)^{2} \cdot c \cdot t^{(\frac{(p-1)^2}{4})}$. Here $c$ equals the determinant of the following Cauchy matrix 
 \begin{equation*}
   \begin{bmatrix}
    1 & \frac{1}{2} & \cdots & \frac{1}{(p-1)/2} \\
    \frac{1}{2} & \frac{1}{3} & \ddots & \vdots\\
    \vdots & \ddots & & \\
    \frac{1}{(p-1)/2} & \cdots & \frac{1}{p-3} & \frac{1}{p-2}
   \end{bmatrix}_{\frac{p-1}{2} \times \frac{p-1}{2}}.
  \end{equation*}
It is non-zero. So the claim is proved. Summarizing the discussions, we see that the term
$$
\lambda^{(\frac{p-1}{2})^{2}}  \cdot (1-\lambda)^{(\frac{p-1}{2})^{2}} \cdot H_{p}(\lambda)
$$
divides $\det A(\lambda)$. Since $\deg H_{p}(\lambda)=\frac{p-1}{2}$ as a matter of fact, the total degree of the term reaches already the upper bound of the degree of $\det A(\lambda)$. So they are equal up to a non-zero constant. The formula \ref{formula for det} is proved. 
\end{proof}

\section{Classification, counting and distribution}
In the final section, we work exclusively over $\mathbb C$. Let $C$ be a smooth projective curve over $\mathbb{C}$, equipped with a reduced effective divisor $D=\sum_{1\leq i\leq l}D_i$. Fix the following numerical data:
\begin{itemize}
    \item [(i)] Positive integer $r$;
    \item [(ii)] For each $1\leq i\leq l$ (when $D\neq \emptyset$), a tuple $\Vec{w}_i=(w_1,\cdots,w_{m_i})$ of rational numbers satisfying 
    $$
    0\leq w_1<\cdots<w_{m_i}<1,
    $$ 
    and a tuple $\Vec{n}_i=(n_1,\cdots,n_{m_i})$ of positive integers satisfying $$\sum_{1\leq j\leq m_i}n_j=r.$$
\end{itemize}
For $(r,\underline{w}=(\vec{w}_1,\cdots,\vec{w}_l),\underline{n}=(\vec{n}_1,\cdots,\vec{n}_l))$ when $D\neq \emptyset$ (and simply $r$ otherwise), let   $M_{r,\underline{w},\underline{n}}$ (simply $M_r$ when $D=\emptyset$) be the moduli scheme over $\mathbb C$, corepresenting the moduli functor of rank $r$ semistable graded parabolic Higgs bundles over $(C,D)$, whose quasi-parabolic structure at $D_i$ is given by $(\vec{w}_i,\vec{n}_i)$ and whose determinant is of torsion. Take a spread-out $(\sC,\sD)/S$ of the pair $(C,D)$. One obtains a corresponding spread-out $\sM_{r,\underline{w},\underline{n}}/S$ (simply $\sM_r$) of the moduli space.  Then there exists a proper closed subset $Z\subset S$, such that for any geometric closed point $s\in S-Z$, and any $W_2(k(s))$-lifting $\tilde s: \Spec\ W_2(k(s))\to S$, there exists a positive integer $\Lambda$, independent of $s,\tilde s$ and the choice of a spread-out for $(C,D)$, such that the $\Lambda$-iteration of the flow operator $\Gr_{Fil_S}\circ C_{s\subset \tilde s}^{-1}$ in \S2 induces a self map on $\sM_{r,\underline{w},\underline{n}}(k(s))$. A point in $M_{r,\underline{w},\underline{n}}(\mathbb C)$ is said to be \emph{periodic}, if it is represented by a periodic parabolic Higgs bundle $(E,\theta)$ over $(C,D)$ (see Definition \ref{globally periodic}). For such an $(E,\theta)$, the set of integers $\{f\}$ in Definition \ref{globally periodic} is generated by its minimal element $f(x)$. We call $f(x)$ the \emph{period} of $x$. It is an important invariant of a periodic parabolic Higgs bundle. Let $\sP{r,\underline{w},\underline{n}}\subset M_{r,\underline{w},\underline{n}}(\mathbb C)$ (and simply $\sP_r\subset M_r$ when $D=\emptyset$) be the set of all periodic points. To study them, one needs to understand the action of $(\Gr_{Fil_S}\circ C_{s\subset \tilde s}^{-1})^{\Lambda}$ on $\sM_{r,\underline{w},\underline{n}}(k(s))$ by default. The main question concerning us is the \emph{classification, counting, and distribution} of periodic points in the moduli space. It has a clear geometric meaning because of the Higgs periodicity theorem (Theorem \ref{Higgs periodicity theorem}).
	 
In \cite[Proposition 2.5]{kshiggs}, it is proved that a rank one Higgs bundle $(L,\theta)$ over $C$ is periodic if and only if $\theta=0$ and $L$ is of torsion \footnote{It is not difficult to extend this classification result to the parabolic case, via the generalized Biswas-Iyer-Simpson correspondence (see \cite[\S2]{ksdR}).}. Therefore, $\sP_{1}$ is exactly the set of torsion points in $\mathrm{Pic}^0(C)$. There are two distinguished properties about this subset: i). For each positive integer $N$, there are only finitely many torsion points whose order $\leq N$; ii). The set of torsion points as a whole is equi-distributed in $\mathrm{Pic}^0(C)$. However for a higher rank moduli, the notion of order is not appropriate any more. Note that in i), one may replace a bounded order with a bounded period. And the notion of period extends. The aim of this section is to investigate this problem for a moduli space of rank two, degree zero (semi)stable graded parabolic Higgs bundles over $(\P^1,D=0+1+\lambda+\infty)$ with \emph{equal parabolic weight}. To be precise, a graded parabolic Higgs bundles $(E,\theta)$ over $(\P^1,D)$ lies in such a moduli, if it satisfies the following conditions:
	\begin{itemize}
		\item [\romannumeral1)] $\rank \, E=2,\quad \deg \, E=0$;
		\item [\romannumeral2)] parabolic weights of $E$ along each point of $D$ are equal;
		\item [\romannumeral3)] $(E,\theta)$ is (semi)stable.
	\end{itemize} 

\subsection{Classification}
We shall not consider the strictly semistable case, since up to $S$-equivalence, it is just a direct sum of two parabolic lines bundles of degree zero. In the following, we assume taciturnly that $(E,\theta)$ is stable. 

Rewrite $D=\sum_iD_i$. Because of condition (ii), we may write $E$ as
	\begin{equation*}
	E=(\OO_{\P^1}(a) \oplus \OO_{\P^1}(b))(\sum_i\alpha_iD_i)
	\end{equation*}
	with $ a \geq b $ and $\alpha_i\in [0,1)\cap \Q$. 
	\begin{lemma}\label{First classification}
		Notation as above. If $(E,\theta)$ is periodic, then the followings hold:
		\begin{itemize}
			\item [\romannumeral1)] $\theta^{1,0}:\OO_{\P^1}(a)(\sum_i\alpha_iD_i)\to \OO_{\P^1}(b)(\sum_i\alpha_iD_i)\otimes \Omega_{\P^1}(D)$ is non-zero;
			\item [\romannumeral2)] $a=b+2$ or $a=b+1$.
		\end{itemize}
	\end{lemma}
	
	\begin{proof}
	By stability, $\theta$ cannot be zero. As $a\geq b$, one has
  $$
  E^{1,0}=\OO(a)(\sum_i\alpha_iD_i),\quad E^{0,1}=\OO(b)(\sum_i\alpha_iD_i).
  $$
 Moreover, $a>b$. By $\theta^{1,0}\neq 0$, it follows that $a\leq b+2$. Therefore, $a=b+2$ or $a=b+1$.
  \end{proof}
By (ii), the classification falls into a dichotomy: $a=b+2$, equivalently $\theta$ is maximal; $a=b+1$, equivalently $\theta$ has a unique zero. 

The simplest example in the former case is given by the uniformizing Higgs bundle of $(\P^1,D)$:
	\begin{equation*}
	\theta_{unif}: \sO_{\Pl}(1) \stackrel{\cong}{\longrightarrow} \sO_{\Pl}(-1)\otimes \Omega_{\Pl}(D).
	\end{equation*}
This case has been studied in \cite{LS}. The following result has been obtained therein. 
	\begin{proposition}[Corollary 1.4, Proposition 1.5 \cite{LS}]\label{LS result}
		The following statements hold:
		\begin{itemize}
			\item [(\romannumeral1)] Let $\lambda$ be one of 17 values in Beauville's list. Then $(E_{unif}=\sO_{\Pl}(1)\oplus \sO_{\Pl}(-1),\theta_{unif})$ is periodic.
			\item [(\romannumeral2)] If $(E_{unif},\theta_{unif})$ is periodic, then $\lambda\in \bar \Q$.
		\end{itemize}
	\end{proposition}
It is easily seen that in former case, we may write
	\begin{equation*}
	(E,\theta) \cong (E_{\mathrm{unif}},\theta_{\mathrm{unif}}) \otimes (\sO_{\Pl}(n)(\sum_i\alpha_iD_i),0),
	\end{equation*}
	where $ (E_{\mathrm{unif}},\theta_{\mathrm{unif}}) $ is equipped with the trivial parabolic structure and $\sO_{\Pl}(n)(\sum_i\alpha_iD_i)$ is a parabolic line bundle of degree zero. A simple computation shows that $ n \in \left\lbrace 0,1,2,3\right\rbrace $. Using Proposition \ref{LS result}, we obtain the following classification result in the case $ a=b+2 $.
	\begin{proposition}\label{case 1}
		Notations as above. Suppose $\lambda$ in Beauville's list. Then
		\begin{equation*}
		(E_{\mathrm{unif}},\theta_{\mathrm{unif}})\otimes (\sO_{\Pl}(-n)(\sum_i\alpha_iD_i),0),
		\end{equation*}
with $\sum\alpha_i=n$ is always periodic. Conversely, suppose there exists some $(E,\theta)$ satisfying $a=b+2$ is periodic. Then it must be of the above form, and $\lambda$ is an algebraic number \footnote{By \cite[Conjecture 1.6]{LS}, $\lambda$ must be one of 17 values in Beauville's list}. 
	\end{proposition}
	\begin{proof}
	The only thing to remind is that $(\sO_{\Pl}(n)(\sum_i\alpha_iD_i),0)$ is periodic if and only if it is of degree zero viz. $\sum\alpha_i=n$ (however it is not one-periodic in general by \cite[Proposition 2.18, Remark 5.9]{ksdR}).
	\end{proof}
	
	In the latter case, we consider the logarithmic Higgs bundle studied by Sun-Yang-Zuo: 
	\begin{equation*}
	\theta_x: \sO_{\Pl} \to \sO_{\Pl}(-1) \otimes \Omega_{\P^1}(D).
	\end{equation*}
	\begin{proposition}\label{case 2}
		Suppose $(E,\theta)$ satisfy $a=b+1$. Then it is periodic if and only if there is an isomorphism of graded parabolic Higgs bundles:
		\begin{equation}\label{classification for a=b+1}
		(E,\theta) \cong (\sO_{\Pl} \oplus \sO_{\Pl}(-1),\theta_{x}) \otimes (\sO_{\Pl}(-n)(\sum_i\alpha_iD_i),0),
		\end{equation}
		where $ x $ is the image of a torsion point in $C_{\lambda}$ under $ \pi:C_{\lambda}\to \P^1$, and $n\in \{0,1,2,3\}$ satisfying $\sum_i\alpha_i=\frac{1}{2}+n$.    
	\end{proposition}
	\begin{proof}
		Equipping the last logarithmic Higgs bundle with the trivial parabolic structure, an $ (E,\theta) $ satisfying $a=b+1$ is written into a tensor product
		\begin{equation}\label{tensor structure of equal weight}
		(E,\theta) \cong (\sO_{\Pl} \oplus \sO_{\Pl}(-1),\theta_{x}) \otimes (\sO_{\Pl}(-n)(\sum_i\alpha_iD_i),0)
		\end{equation}  
		with $\pardeg\ \sO_{\Pl}(-n)(\sum_i\alpha_iD_i)=\frac{1}{2}$. Then by Lemma \ref{tensor product lemma}, $(E,\theta)$ is periodic if and only if $( \sO_{\Pl}\oplus \sO_{\Pl}(-1),\theta_{x})$ is periodic in the sense of Sun-Yang-Zuo. By Theorem \ref{main result} and the proof of Theorem \ref{torsion theorem}, the result follows.
		
	\end{proof}

	\subsection{Counting periodic points} 
The counting problem for the case in Proposition \ref{case 1} is not so interesting. We shall focus on the second case, which has a positive dimensional moduli. Let $M_{\alpha}$ be the moduli scheme introduced in Proposition \ref{identification of self-maps}, which is identical to the one for graded parabolic Higgs bundles in Proposition \ref{case 2}. Let us tackle first the simplest case $M_{\frac{1}{2}P}$. By Proposition \ref{identification of self-maps}, for $p\geq 3$, we have the following commutative diagram:

\begin{equation*} 
			\begin{tikzcd}
				M_{\mathrm{Higgs}}(\bar \F_p) \arrow[r,"\varphi_{\lambda,p}"] 
				\arrow[d,"\cong "'] & M_{\mathrm{Higgs}}(\bar \F_p)
				\arrow[d,"\cong"] \\
				M_{\frac{1}{2}P}(\bar \F_p) \arrow[r,"\varphi_{\frac{1}{2}P}"] & M_{\frac{1}{2}P}(\bar \F_p),
			\end{tikzcd}  
		\end{equation*}	
Thanks to the commutative diagram \ref{commutitivity of flow operator}, we may transform the problem to a counting of torsion points on $\Sigma_{\Higgs}$.

For each positive integer $N$, let $U_N$ be the unit group of $\Z/N\Z$. The universal exponent of $U_N$ is the smallest exponent $e$ such that the morphism $U_N\to U_N, \quad x\mapsto x^e$ has trivial image. We denote it by $\lambda(N)$. Write 
\begin{equation*}
N=p_0^{s_0}\cdot p_{1}^{s_{1}} \cdots p_{n}^{s_n} ,
\end{equation*} 
($p_0=2$) for the prime decomposition. Then
\begin{equation*}
	\lambda(N)= \mathrm{lcm} \left\lbrace 2^{\max \left\lbrace 0, s_0-2\right\rbrace }, (p_1-1)p_{1}^{s_{1}-1}, \cdots , (p_n-1)p_{n}^{s_{n}-1})\right\rbrace .
	\end{equation*}
Identifying $M_{\frac{1}{2}P}$ with $\P^1$ as above, Proposition \ref{case 2} tells us that its subset of periodic points is the image of the set of torsion points of $C_{\lambda}$ under the projection $\pi$. 
\begin{proposition}\label{torsion order to period}
 Let $x \in M_{\frac{1}{2}P}(\mathbb C)$ be the image of a strict $N$-torsion point in $C_{\lambda}$. Then its period $f(x)$ equals 
 \begin{equation*}
		\lambda'(N):= \begin{cases}
		\frac{\lambda(N)}{2},\quad & \text{if a nontrivial primitive root exists modulo}\  N, \\
		\lambda(N),\quad & \text{otherwise}.
		\end{cases} 
		\end{equation*}
 \end{proposition}
\begin{proof}
By Proposition \ref{intersection points are fixed}, the statement holds for two torsion points. We assume $N>2$ below. Note that $\lambda(N)$ is then even. By Theorem \ref{main result}, it follows that the period $f(x)$ of $x$ is the smallest positive integer satisfying the equality 
\begin{equation*}\label{torsion and periodic}
		p^{f(x)} \equiv \pm 1 \mod N,
\end{equation*}
for almost all primes $p$. By Dirichlet density theorem, it follows that when $U_N$ is cyclic, $f(x)=\frac{\lambda(N)}{2}$. Otherwise, since for any proper divisor $d$ of $ \lambda(N)$, there exists some element $g$ such that $g^{d} \neq \pm 1$, it follows that $f(x)=\lambda(N)$ in the remaining case. 		 
\end{proof}
For $N\geq 2$, define $\phi_2(N)=N^2 \prod_{i=0}^{n}(1-\frac{1}{p_{i}^{2}})$, which counts the number of strict $N$-torsion points in $(\Z/N\Z)^2\cong C_{\lambda}[N]$.
\begin{corollary}\label{counting for M_{1/2P}}
For any $f\in \mathbb{N}$, the number of $f$-periodic points in the moduli $M_{\frac{1}{2}P}$ equals 
\begin{equation*}
		Z(f)=\begin{cases}
  26, \quad & \text{if } f=1,\\
		\sum_{N\in \lambda'^{-1}(f)}\frac{\phi_2(N)}{2}, \quad & \text{if } f>1, \\
\end{cases}
\end{equation*} 
\end{corollary}
\begin{proof}
The statement is clear, except for the case $f=1$. It is easy to see that, other than the four 2-torsion points of $C_{\lambda}$, the images of strict $3,4,6$-torsion points also contribute to the set of one-periodic points, and there are no others. Therefore, 
$$
Z(1)=4+\frac{\phi_{2}(3) +\phi_{2}(4) +\phi_{2}(6)}{2}=26.
$$
This completes our proof.
\end{proof}	

\begin{remark}
The formula on number of $f$-periodic points is known to Jinbang Yang. The number 26 first appeared in Kang Zuo's talk at Lyon \cite{Zuo}. By the formula, one sees that \emph{not} every natural number appears as a period. The first such number is 7. 
\end{remark}

Write $\alpha=(\alpha_1,\cdots,\alpha_4)$, and $\alpha_i=\frac{m_i}{N_i}$, where $m_i=0, N_i=1$ if $\alpha_i=0$, and $(m_i,N_i)=1$ otherwise. Set $N=\mathrm{lcm} \left\lbrace N_1,\cdots,N_4\right \rbrace$. Recall from Proposition \ref{identification of self-maps}, the flow operator in characteristic $p$ with $(p,N)=1$ induces the map
$$
\varphi_{\alpha}: M_{\alpha}(\bar \F_p)\to M_{\langle p\alpha\rangle}(\bar \F_p).
$$
\begin{lemma}
The least positive integer $\Lambda_{\alpha}$ such that $\varphi_{\alpha}^{\Lambda_{\alpha}}$ is a self-map of $M_{\alpha}(\bar \F_p)$ equals $\lambda(N)$.
\end{lemma}
\begin{proof}
Suppose $\alpha_1=\frac{m_1}{N_1}\neq 0$. As $(m_1,N_1)=1$, the least positive integer $f$ such that the equality
$$
p^{f}m_1 \equiv m_1 \mod N_1
$$
holds for almost all primes $p$ equals $\lambda(N_1)$ by the proof of Proposition \ref{torsion order to period}. Hence
$$
\Lambda_{\alpha}=\mathrm{lcm} \left\lbrace \lambda(N_1),\cdots,\lambda(N_4)\right \rbrace=\lambda(N),
$$
as claimed.
\end{proof}
Thus the period of a periodic point in $M_{\alpha}$ must be a multiple of $\Lambda_{\alpha}$. For $f\in \N$, let 
$$
\Lambda_{\alpha}=\prod_{i=1}^np_i^{s(\Lambda_{\alpha})_i}, \quad f=\prod_{i=1}^np_i^{s(f)_i},
$$ 
where
$$
s(\Lambda_{\alpha})_i \geq 0, \ s(f)_i\geq 0,\  \max\{s(\Lambda_{\alpha})_i,s(f)_i\}>0.
$$
Set $\textrm{D}(f)=\{p_i,\ p_i|f\}\subset \{p_1,\cdots,p_n\}$.
\begin{proposition}\label{counting periodic points for M_alpha}
Notation as above. Then the number of $\Lambda_{\alpha}\cdot f$-periodic points in $M_{\alpha}$ equals
	\begin{equation*}
	\sum_{p_i\notin \textrm{D}(f), \atop 0\leq j_i\leq s(\Lambda_{\alpha})_i}Z\left(\frac{\Lambda_{\alpha}\cdot f}{p_i^{j_i}}\right).
	\end{equation*}
\end{proposition}	
\begin{proof}
Consider the composite of the natural identifications in Proposition \ref{identification of self-maps}:
$$
\beta: M_{\frac{1}{2}P}\stackrel{\cong}{\to} M_{\Higgs}\stackrel{\cong}{\to} M_{\alpha}. 
$$
Let $\sP_{\frac{1}{2}P}$ (resp. $\sP_{\alpha}$) be the subset of periodic points in $M_{\frac{1}{2}P}$ (resp. $M_{\alpha}$). Then $\beta: \sP_{\frac{1}{2}P}\to \sP_{\alpha}$ is a bijection. Moreover, for a periodic point $x$ in $M_{\frac{1}{2}P}$, one has 
$$
f(\beta(x))=\mathrm{lcm} \left\lbrace \Lambda_{\alpha}, f(x)\right \rbrace.
$$
Thus when and only when $f(x)=\prod_i p_i^{t_i}$, where
\begin{equation*}
		t_i=\begin{cases}
  s(f)_i+s(\Lambda_{\alpha})_i, \quad & \text{if } s(f)_i\neq 0,\\
		0,\cdots,s(\Lambda_{\alpha})_i,\quad & \text{if } s(f)_i=0,\\
\end{cases}
\end{equation*} 
$f(\beta(x))=\Lambda_{\alpha}\cdot f$. Then the result follows. 
\end{proof}

	\subsection{Equidistribution of periodic points}
The final subsection concerns the equidistribution property of $\sP_{\alpha}$ inside  $M_{\alpha}$. 
	\begin{definition}
		A polarized algebraic dynamical system over a field $ K $ consists of a triple $ (X, f, L) $ where:
		\begin{itemize}
			\item $ X $ is a projective variety over $K$,
			\item $ f : X \to X $ is an algebraic morphism over $ K $,
			\item $ L $ is an ample line bundle on $X$ polarizing $f$ in the sense that $f^{*}L \cong L^{\otimes q}$ for some positive integer $q$.
		\end{itemize}
	\end{definition}
	
We refer to \cite{Zhang2005DistributionsIA} for the notion of  an \emph{equilibrium measure} $ \mu_{f} $ and a \emph{canonical height} $ h_{f} $, which are associated to a polarized algebraic dynamical system $ (X, f, L) $. 
\begin{definition}
		Let $\left\lbrace x_{m}\right\rbrace_{m \geq 1}$ be an infinite sequence of $ \overline{K} $ points of $X$.
		\begin{itemize}
			\item The sequence is called equidistributed if the probability measure $ \mu_{x_{m}} $ converges weakly to the equilibrium measure $ \mu_{f} $.
			\item The sequence is called generic if any infinite subsequence is Zariski dense in $ X $.
			\item The sequence is called $ h_{f} $ -small if $ h_{f} (x_{m}) \to 0 $ as $m \to \infty$.
		\end{itemize}
	\end{definition}
Roughly speaking, the equidistribution theorem asserts that a generic and small sequence is equidistributed. For each positive integer $m$, the multiplication by $m$ on $\Sigma_{\Higgs}=C_{\lambda}$ descends to a morphism over $M_{\Higgs}$, which in turn induces a morphism over $M_{\alpha}$ via the natural identification in Proposition \ref{identification of self-maps}. We denote it by $f_m: M_{\alpha}\to M_{\alpha}$. 

\begin{proposition}
Notations as above. Then for the polarized algebraic dynamical system $(M_{\alpha},f_m,\sO(1))$, the set $\sP_{\alpha}$ of periodic points is equidistributed with respect to the standard Fubini-Study measure $\mu_{FS}$.
\end{proposition}
\begin{proof}
Note that the set $\textrm{Prep}(f_m)$ of preperiodic points in $M_{\alpha}$ with respect to $f_m$ is exactly the image of the set of preperioic points in $\Sigma_{\Higgs}$ with respect to $[m]$, which is the set of torsion points. Hence $\sP_{\alpha}=\textrm{Prep}(f_m)$. But since any preperiodic point is $h_{f_m}$-small (actually zero), the result follows from \cite[Theorem 3.1]{Yuan2008BigLB} when $ \mathbb{Q}(\lambda) $ is a number field and \cite[Theorem 6.1]{Moriwaki1998ArithmeticHF} when $ \mathbb{Q}(\lambda) $ is a function field.
\end{proof}

\end{document}